\documentclass[11pt]{article}

\usepackage{amssymb, amsmath, amsthm, graphicx}
\usepackage[left=1in,top=1in,right=1in]{geometry}
\date{}
\usepackage{subfigure}
\theoremstyle{plain}
      \newtheorem{theorem}{Theorem}[section]
      \newtheorem{lemma}[theorem]{Lemma}

      \newtheorem{corollary}[theorem]{Corollary}

\theoremstyle{definition}
      \newtheorem{definition}[theorem]{Definition}

\theoremstyle{remark}

\def\ocn{\mbox{\rm odd-cr}}

\title{A structure theorem for pseudo-segments and its applications}
\author{Jacob Fox\thanks{Stanford University, Stanford, CA. Supported by a Packard Fellowship and by NSF award DMS-1855635. Email: {\tt jacobfox@stanford.edu.}} \and J\'anos Pach\thanks{R\'enyi Institute of Mathematics, H-1364 Budapest, POB 127, Hungary.  Supported by NKFIH grants K-131529, Austrian Science Fund Z 342-N31, and ERC Advanced Grant 882971``GeoScape.''Email:
{\tt pach@cims.nyu.edu}.}\and  Andrew Suk\thanks{Department of Mathematics, University of California at San Diego, La Jolla, CA, 92093 USA. Supported an NSF CAREER award and NSF award DMS-1952786. Email: {\tt asuk@ucsd.edu}.} }

\begin{document}

\maketitle

\begin{abstract}
We prove a far-reaching strengthening of Szemer\'edi's regularity lemma for intersection graphs of pseudo-segments. It shows that the vertex set of such a graph can be partitioned into a bounded number of parts of roughly the same size such that almost all bipartite graphs between different pairs of parts are \emph{complete} or \emph{empty}. We use this to get an improved bound on disjoint edges in simple topological graphs, showing that every $n$-vertex simple topological graph with no $k$ pairwise disjoint edges has at most $n(\log n)^{O(\log k)}$ edges. 

\end{abstract}

\section{Introduction}

Given a set of curves $\mathcal{C}$ in the plane, we say that $\mathcal{C}$ is a collection of \emph{pseudo-segments} if any two members in $\mathcal{C}$ have at most one point in common, and no three members in $\mathcal{C}$ have a point in common. The intersection graph of a collection $\mathcal{C}$ of sets has vertex set $\mathcal{C}$ and two sets in $\mathcal{C}$ are adjacent if and
only if they have a nonempty intersection. 

A partition of a set is an \emph{equipartition} if each pair of parts in the partition differ in size by at most one. Szemer\'edi's celebrated regularity lemma roughly says that the vertex set of any graph has an equipartition such that the bipartite graph between almost all pairs of parts is random-like. 

Our main result is a strengthening of Szemer\'edi's regularity lemma for intersection graphs of pseudo-segments. It replaces the condition that the bipartite graphs between almost all pairs of parts is random-like to being homogeneous, either complete or empty. 

\begin{theorem}\label{regularity}
For each $\varepsilon>0$ there is $K=K(\varepsilon)$ such that for every finite collection $\mathcal{C}$ of pseudo-segments in the plane, there is an equipartition of $\mathcal{C}$ into $K$ parts $\mathcal{C}_1,\ldots,\mathcal{C}_K$ such that for all but at most $\varepsilon K^2$ pairs $\mathcal{C}_i$, $\mathcal{C}_j$ of parts, either every curve in $\mathcal{C}_i$ crosses every curve in $\mathcal{C}_j$, or every curve in $\mathcal{C}_i$ is disjoint from every curve in $\mathcal{C}_j$.   
\end{theorem}

Pach and Solymosi \cite{solymosi2} proved the special case of Theorem \ref{regularity} where $\mathcal{C}$ is a collection of segments in the plane, and this result was later extended to semi-algebraic graphs \cite{APPRS} and hypergraphs \cite{FPS16} of bounded description complexity. However, the techniques used to prove these results heavily rely on the algebraic structure. In fact,
while it follows from the Milnor-Thom theorem that there are only $2^{O(n\log n)}$ graphs on $n$ vertices which are semialgebraic of bounded description complexity (see \cite{solymosi,APPRS,Sauermann}) there are many more (namely $2^{\Omega(n^{4/3})}$) graphs on $n$ vertices which are intersection graphs of pseudo-segments \cite{fps+}. 

Theorem~\ref{regularity} does not hold if we do not allow to have exceptional pairs of parts. Indeed, the so-called half-graph (i.e., the graph on the vertex set $\{u_i, v_i: 1\le i\le n\}$ with $u_iv_j\in E(G)$ if and only if $i<j$) can be represented as the intersection graph of segments, and it is easy to see that it has no equipartition without exceptional pairs~\cite{MS14}.

Next, we discuss an application of Theorem \ref{regularity} 
in graph drawing.
\medskip

\noindent {\bf Disjoint edges in simple topological graphs.} A \emph{topological graph} is a graph drawn in the plane such that its vertices are represented by points
 and its edges are represented by nonself-intersecting arcs connecting the corresponding points. The edges are
allowed to intersect, but they may not intersect vertices apart from their endpoints.  Furthermore, no two edges are tangent, i.e., if two edges share an interior point, then they must properly cross at that point in common.  A topological graph is \emph{simple} if every pair of its edges intersect at most once.  Two edges of a topological graph \emph{cross} if their interiors share a point, and are \emph{disjoint} if they neither share a common vertex nor cross.

A \emph{thrackle} is a simple topological graph with no two disjoint edges.  A famous conjecture due to Conway states that every $n$-vertex thrackle has at most $n$ edges.  In fact, Conway offered a \$1000 reward for a proof or disproof of this conjecture.  The first linear upper bound was established by  Lov\'asz, Pach, and Szegedy in \cite{lovasz}, who showed that every $n$-vertex thrackle has at most $2n$ edges.  This bound was subsequently improved in \cite{cairns1,fulek,fulek2}.   

Determining the maximum number of edges in a simple topological graph with no $k$ pairwise disjoint edges seems to be a difficult task. In \cite{pach2},  Pach and T\'oth showed that every $n$-vertex simple topological graph with no $k\geq 3$ pairwise disjoint edges has at most $O(n\log^{4k-8}n)$ edges.  They conjectured that for every fixed $k$, the number of edges in such graphs is at most $O_k(n)$. Our next result substantially improves the upper bound for large $k$.

\begin{theorem}
\label{disjoint}
If $G = (V,E)$ is an $n$-vertex simple topological graph with no $k$ pairwise disjoint edges, then $|E(G)| \leq n(\log n)^{O( \log k)}$. 
\end{theorem}

In \cite{fox2}, Fox and Sudakov  showed that every dense $n$-vertex simple topological graph contains $\Omega(\log^{1 + \delta}n)$ pairwise disjoint edges, where $\delta \approx 1/40$.  As an immediate Corollary to Theorem \ref{disjoint}, we improve this bound to nearly polynomial under a much weaker assumption.

\begin{corollary}
Let $\varepsilon>0$, and let $G = (V,E)$ be an $n$-vertex simple topological graph with at least $2n^{1+\varepsilon}$ edges. Then $G$ has $n^{\Omega(\varepsilon /\log\log n)}$ pairwise disjoint edges.
\end{corollary}

\noindent For complete $n$-vertex simple topological graphs, Aichholzer et al.~\cite{aich} showed that one can always find $\Omega(n^{1/2})$ pairwise disjoint edges.  Whether or not this can be improved to $\Omega(n)$ is an open problem.
\smallskip

The proofs of the above theorems heavily rely on a bipartite Ramsey-type result for intersection graphs of pseudo-segments (Theorem~\ref{main}), as will be explained in Section~\ref{her}. At the end of Section \ref{her},  using a variant of Szemer\'edi's regularity lemma originally proposed by Koml\'os and S\'os~\cite{KS96, mat, pach98}, we show how this Ramsey-type result implies two ``density-type'' theorems (Theorems~\ref{old} and \ref{appdis}). Roughly speaking, we prove that if there are many (resp., few) crossings between two families of pseudo-segments, then they have two large subfamilies such that every member of the first subfamily crosses every member of the second (resp., no member of the first subfamily crosses any member of the second).  In Subsection \ref{strongeh}, we show that any finite collection of pseudo-segments in the plane contains a linear-sized subset with the property that only a small fraction of pairs in the subset are crossing, or nearly all of them cross. In Section~\ref{dbground}, we prove Theorem \ref{main} in the special case where one of the families is double grounded.  Building on these results, in Section \ref{highvslow}, we establish our bipartite Ramsey-type theorem (Theorem \ref{main}) for any two families of pseudo-segments with the property that for each family, only a small fraction of pairs are crossing, or nearly all of them cross.  Finally, in Section \ref{pfmain}, we prove Theorem \ref{main} in its full generality. We conclude our paper with a number of remarks and with an application of our results to an old problem in graph drawing (Theorem \ref{disjoint}).

\section{Ramsey-type properties of hereditary families of graphs}\label{her}

A family of graph is \emph{hereditary} if it is closed under taking induced subgraphs. In this section, we show that the property that in a hereditary family of graphs, a homogeneous regularity lemma holds, is equivalent to some seemingly weaker Ramsey-type properties. See Theorem~\ref{equivalences}, below. It will imply that in order to prove Theorem \ref{regularity}, it is enough to establish that the family of intersection graphs of pseudo-segments satisfies some bipartite Ramsey-type condition (Theorem~\ref{main}). 

We need the following definitions. 

\begin{definition} Let $\mathcal{F}$ be a  family of graphs. 
\begin{enumerate} 
\item $\mathcal{F}$ has the \emph{Erd\H{o}s-Hajnal property} if there is $\varepsilon=\varepsilon(\mathcal{F})>0$ such that every graph in $\mathcal{F}$ with $n$ vertices contains a clique or independent set of size $n^{\varepsilon}$.
\item  $\mathcal{F}$ has the \emph{polynomial R\"odl property} if there is $C=C(\mathcal{F})$ such that for every $\varepsilon>0$, every $n$-vertex graph in $\mathcal{F}$ has an induced subgraph on at least $\epsilon^{C(\mathcal{F})}n$ vertices that is $\varepsilon$-homogeneous. 
\item $\mathcal{F}$ has the \emph{strong Erd\H{o}s-Hajnal property} if there is a constant $\varepsilon=\varepsilon(\mathcal{F})>0$ such that the vertex set of every $n$-vertex graph in $\mathcal{F}$ with $n \geq 2$ has two disjoint subsets, each of size at least $\varepsilon n$, such that the bipartite graph between them is complete or empty. 
\item $\mathcal{F}$ has the \emph{mighty Erd\H{o}s-Hajnal property} if there is a constant $\varepsilon=\varepsilon(\mathcal{F})>0$ such that for every graph $G$ in $\mathcal{F}$ and every pair of disjoint subsets $A,B \subset V(G)$ with $|A|=|B|$, there are $A' \subset A$ and $B' \subset B$ with $|A'| \geq \varepsilon |A|$ and $|B'| \geq \varepsilon |B|$ such that the bipartite graph between $A$ and $B$ in $G$ is complete or empty. 
\item $\mathcal{F}$ has the \emph{homogeneous regularity property} if for every $\varepsilon>0$ there is $K=K_{\mathcal{F}}(\varepsilon)$ such that the vertex set of every graph $G \in \mathcal{F}$ has an equipartition 
$V(G)=V_1 \cup \cdots \cup V_K$ into $K$ parts such that for all but at most $\varepsilon K^2$ pairs of parts $(V_i,V_j)$, the bipartite graph between them is complete or empty (i.e., $V_i \times V_j \subset E(G)$ or $V_i \times V_j \cap E(G) = \emptyset$). 
\item $\mathcal{F}$ has the \emph{homogeneous density property} if for every $\varepsilon>0$ there is $C=C(\varepsilon)$ such that for every $G \in \mathcal{F}$ and every pair $A,B \subset V(G)$ of disjoint subsets of the same size with at least $\varepsilon |A||B|$ edges between them, there are $A' \subset A$ and $B' \subset B$ with $|A'| \geq \varepsilon^{C}|A|$ and $|B'| \geq \varepsilon^C |B|$ with the property that $A'\times B'\subset E(G).$ 
\end{enumerate}
\end{definition}

For any family of graphs, $\mathcal{F}$, let $\overline{\mathcal{F}}$ be the family consisting of the complements of the elements of $\mathcal{F}$. In the sequel, we will heavily use the following.

\begin{theorem}\label{equivalences}
Let $\mathcal{F}$ be a hereditary family of graphs. Then the following are equivalent  
\begin{enumerate}
    \item $\mathcal{F}$ has the mighty Erd\H{o}s-Hajnal property. 
    \item $\mathcal{F}$ has the homogeneous regularity property. 
    \item $\mathcal{F}$ and $\overline{\mathcal{F}}$ both have the homogeneous density property. 
\end{enumerate}
\end{theorem}

Before proving Theorem \ref{equivalences}, we make some remarks about the relative strengths of the above properties. 

A classical result of Erd\H{o}s and Szekeres~\cite{ErSz35} states that every graph on $n$ vertices has a clique or independent set on $\frac{1}{2}\log_2 n$ vertices. The first improvement on the constant $\frac12$ was found recently by Campos, Griffiths, Morris, and Sahasrabudhe \cite{CGMS}. From the other direction, Erd\H{o}s~\cite{Er47} showed that for every integer $n>2$, there is a graph on $n$ vertices that contains no clique or independent set on $2\log_2 n$ vertices. 

It is known that this result can be substantially improved in several restricted settings; see, e.g., \cite{Chudnovsky,FP08}. According to the famous \emph{Erd\H{o}s-Hajnal conjecture} \cite{EH89}, which has been established only in some special cases, every hereditary family of graphs that is not the family of all graphs, satisfies the Erd\H{o}s-Hajnal property. 

R\"odl \cite{Rodl} proved that for any hereditary family $\mathcal{F}$ of graphs that is missing a graph and any $\varepsilon>0$, there is $\delta=\delta_{\mathcal{F}}(\varepsilon)>0$ such that any graph in $\mathcal{F}$ on $n$ vertices contains an induced subgraph on at least $\delta n$ vertices which is $\varepsilon$-homogeneous. R\"odl's proof used Szemer\'edi's regularity lemma and consequently gives a weak quantitative bound. Better quantitative bounds were obtained more recently in \cite{fs08,bnss}. 

Fox and Sudakov \cite{fs08} conjectured that a polynomial dependence holds in R\"odl's theorem. That is, every hereditary family of graphs that is not the family of all graphs, satisfies the polynomial R\"odl property. 
It is a simple exercise that any family having the polynomial R\"odl property also has the Erd\H{o}s-Hajnal property. In particular, the conjecture of Fox and Sudakov is a strengthening of the Erd\H{o}s-Hajnal conjecture. Many cases of this strengthening of the Erd\H{o}s-Hajnal conjecture were recently verified in \cite{fnss,nss}. 

One way to prove that a family of graphs satisfies the Erd\H{o}s-Hajnal property is to show that it has the \emph{strong} Erd\H{o}s-Hajnal property. However, this technique is limited, as the strong Erd\H{o}s-Hajnal property is strictly stronger than the Erd\H{o}s-Hajnal property. For example, it was proved by Tomon~\cite{tomon} that the family of intersection graphs of curves in the plane (the family of \emph{``string graphs''}) has the Erd\H{o}s-Hajnal property, yet, as was shown by Pach and G. T\'oth \cite{PT}, it does not have the strong Erd\H{o}s-Hajnal property. Interestingly, the family of  intersection graphs of curves in the plane, each pair of which have a \emph{bounded number of intersections}, has the strong Erd\H os-Hajnal property \cite{fpt}. The same is true for intersection graphs of $x$-monotone curves or convex sets in the plane \cite{fpt10}. 

In Lemma \ref{strongEHimpliesPR} below, we prove that every hereditary family of graphs that has the strong Erd\H{o}s-Hajnal property must also have the polynomial R\"odl property. The strong Erd\H{o}s-Hajnal property is strictly stronger than the polynomial R\"odl property, as it is a simple exercise to show that the family of triangle-free graphs has the polynomial R\"odl property but does not have the strong Erd\H{o}s-Hajnal property.

However, for some applications, even the strong Erd\H{o}s-Hajnal property is not sufficient, and one needs to introduce a stronger property: the \emph{mighty} Erd\H{o}s-Hajnal property, which is a bipartite version of the strong Erd\H os-Hajnal property. The simplest example of a class of graphs that has the strong property, but not the mighty one, is the family of bipartite graphs. A more interesting example is the family of intersection graphs of \emph{convex sets} in the plane, which has the strong Erd\H os-Hajnal property by \cite{fpt10}, but not he mighty Erd\H os-Hajnal property as it contains the complement of every bipartite graph.  On the other hand, it was proved by Pach and Solymosi \cite{solymosi,solymosi2} that intersection graphs of segments, or, more generally, intersection graphs of semialgebraic sets of bounded description complexity, have even the mighty Erd\H{o}s-Hajnal property \cite{APPRS}.

\subsection{Proof of Theorem \ref{equivalences}}
We split the proof of the equivalences in Theorem \ref{equivalences} into four smaller lemmas. 

\begin{lemma}
Suppose that a hereditary family $\mathcal{F}$ of graphs has the homogeneous regularity property. Then $\mathcal{F}$ also has the mighty Erd\H{o}s-Hajnal property. \end{lemma}
\begin{proof} 
Let $G \in \mathcal{F}$ and $A,B\subset V(G)$ be disjoint with $|A|=|B|=n$. Let $G'$ be the induced subgraph of $G$ with vertex set $A \cup B$. Since $\mathcal{F}$ is hereditary and $G'$ is an induced subgraph of $G \in \mathcal{F}$, then $G' \in \mathcal{F}$. By the homogeneous regularity property of $G'$ applied with $\varepsilon=1/65$ there is an absolute constant $K > 100$ and an equipartition of $A \cup B$ into $K$ parts $V_1,\ldots,V_K$ such that all but $\varepsilon K^2$ pairs of parts $(V_i,V_j)$ are complete or empty. Note that the parts $V_i$ that have at most $n/(2K)$ elements in $A$ together have at most $K \cdot n/(2K)=n/2$ elements from $A$. Each part $V_i$ has at most $\lceil 2n/K \rceil \leq 4n/K$ elements. So there at least $(n/2)/(4n/K)=K/8$ parts $V_i$ that have at least $n/(2K)$ elements from $A$. Similarly, at least $K/8$ parts have at least $n/(2K)$ elements from $B$. So there are $K^2/64$ pairs of parts $(V_i,V_j)$ such that $V_i$ has at least $n/(2K)$ elements from $A$ and $V_j$ has at least $n/(2K)$ elements from $B$. Since we chose $\varepsilon=1/200$, there are at least $K^2/64-K^2/65 > 0$ pairs of parts $(V_i,V_j)$ that are complete or empty to each other, $V_i$ contains at least $n/(2K)$ elements from $A$ and $V_j$ contains at least $n/(2K)$ elements of $B$.  Fixing such a pair $(V_i,V_j)$ and letting $A'=V_i \cap A$ and $B'=V_j \cap B$, we have $A'$ and $B'$ are complete or empty to each other and are of the desired size. So $\mathcal{F}$ has the mighty Erd\H{o}s-Hajnal property. 
\end{proof}

\begin{lemma}
Suppose that a hereditary family $\mathcal{F}$ of graphs has the mighty Erd\H{o}s-Hajnal property. Then $\mathcal{F}$ also has the homogeneous regularity property. \end{lemma}
\begin{proof}

Since $\mathcal{F}$ has the mighty Erd\H{o}s-Hajnal property, there is $c>0$ such that for every graph $G_0$ in $\mathcal{F}$ and for all disjoint vertex subsets $A,B \subset V(G_0)$ with $|A|=|B|$, there is $A' \subset A$ and $B' \subset B$ with $|A'| \geq c|A|$ and $|B'| \geq c|B|$ such that the pair $(A',B')$ is homogeneous. 

We will show that $\mathcal{F}$ also has the homogeneous regularity property. Let $G \in \mathcal{F}$ have $n$ vertices. We will show that $G$ has an equipartition into $K$ parts such that the bipartite graph between all but at most $\varepsilon K^2$ pairs of parts are complete or empty. If $n \leq K$, then partition $V(G)$ into $K$ sets which have size at most one. Then every pair of parts is complete or empty between them. So we may assume $n>K$. 

Arbitrarily equipartition $V(G)=V_1 \cup \ldots \cup V_t$ into $t=4/\varepsilon$ sets, so at most a fraction $\varepsilon/4$ of the pairs of vertices are in the same part in the partition. Delete one vertex from each part $V_i$ that has size $\lfloor n/t \rfloor+1$ so all parts have the same size. Still, all but at least an $\varepsilon/3$ fraction of the pairs of vertices go between these pairs of sets of the same. Let $\mathcal{B}_0=\{(V_i,V_j)\}_{1 \leq i < j \leq t}$. 

Let $\varepsilon'=c^2\varepsilon/100$ At each step $i$ of a process we will have a family $\mathcal{B}_i$  of pairs of disjoint vertex subsets all of the same size satisfying the following four properties: 
\begin{enumerate} 
\item $|\mathcal{B}_i| \leq t^2(1/\varepsilon')^{2i}$. 
\item Each pair of vertices of $G$ go between at most one pair of vertex subsets of $\mathcal{B}_i$. 
\item All but a fraction $(1-2\varepsilon')^{2i}$ of the pairs of vertices that go between a pair of vertex subsets in $\mathcal{B}_0$ go between a pair of vertex subsets in $\mathcal{B}_i$. 
\item At most a $(1-c^2)^i$ fraction of the pairs of vertices of the graph go between a pair of vertex subsets in $\mathcal{B}_i$ that are not homogeneous. 
\end{enumerate} 

Note that the above properties hold for $i=0$. Suppose we have already found $\mathcal{B}_i$. For each pair $(A,B) \in \mathcal{B}_i$, by the mighty Erd\H{o}s-Hajnal property, there are subsets $A' \subset A$ and $B' \subset B$ with $|A'| \geq c|A|$ and $|B'| \geq c|B|$ and the bipartite graph between $A$ and $B$ is complete or empty. Partition each of the four sets $A'$, $A \setminus A'$, $B'$ and $B \setminus B'$ into parts of size $\varepsilon' |A|$ with one possible remaining set. For each pair of these subsets of $A$ and $B$ of size $\varepsilon'|A|$ such that one is a subset of $A$ and the other is a subset of $B$, we place the pair of subsets in $\mathcal{B}_{i+1}$. 

We next observe that each of the four properties listed hold. For each pair $(A,B) \in \mathcal{B}_i$, we get at most $(1/\varepsilon')^2$ pairs in $\mathcal{B}_{i+1}$, so $|\mathcal{B}_{i+1}| \leq (1/\varepsilon')^2|\mathcal{B}_i|$, which inductively gives the first claimed property. The second claimed property holds since every pair of vertices that go between a pair of vertex subsets in $\mathcal{B}_{i+1}$ also go between a pair of vertex subsets in $\mathcal{B}_i$. At least a $(1-2\varepsilon')^2$-fraction of the pairs of vertices in $A \times B$ go between pairs of vertices in $\mathcal{B}_{i+1}$, which inductively gives the third claimed property. Finally, for each $(A,B) \in \mathcal{B}_i$, since at least a $(1-c)^2$ fraction of the pairs of vertices go between the homogeneous pair $(A',B')$, the last claimed property holds. 

After $i_0=c^{-2}\log(4/\varepsilon)$ steps, all but a fraction $(1-c^2)^{i_0}<e^{-c^2 i_0} =\varepsilon/4$ of the pairs of vertices in $G$ go between a pair of vertex subsets in $\mathcal{B}_{i_0}$ that 
are not homogeneous. The fraction of pairs of vertices in $G$ that do no go between a pair of vertex subsets in $\mathcal{B}_{i_0}$ is at most $\varepsilon/3 + 1- (1-2\varepsilon')^{2i_0} \leq \varepsilon/2$.

For each pair $(A,B) \in \mathcal{B}_{i_0}$, we obtain two bipartitions $V(G)=A \cup (V(G) \setminus A)$ and $V(G)=B \cup (V(G) \setminus B)$. Let $\mathcal{Q}$ be the common refinement of all of these $2|\mathcal{B}_{i_0}|$ bipartitions, so $\mathcal{Q}$ is a partition of $V(G)$ into $2^{2|\mathcal{B}_{i_0}}|$ parts.

We next obtain an equipartition $\mathcal{P}$ of $V(G)$ into $K$ parts. For every equipartition into $K$ parts, we have parts of size $\lfloor n/K \rfloor$ or $\lceil n/K \rceil$, with the number of parts of size $\lceil n/K \rceil$ being the remainder when $n$ is divided by $K$. Partition each part of $\mathcal{Q}$ into parts of size $\lfloor n/K \rfloor$ or $\lceil n/K \rceil$, with additional one part of size at most $n/K$. We partition the union of the $|Q|$ additional parts into sets of size $\lfloor n/K \rfloor$ or $\lceil n/K \rceil$ so that to obtain an equipartition $P$ of $V(G)$ into $K$ parts. The fraction of pairs of vertices in $G$ not going between a homogeneous pair in partition $P$ is at most $\varepsilon/2 +\varepsilon/4+2|Q|/K \leq \varepsilon$. Thus, $P$ is the desired equipartition. 
\end{proof}

\begin{lemma}\label{hdpimplymighty}
Suppose that $\mathcal{F}$ and $\overline{\mathcal{F}}$ are hereditary families that both have the homogeneous density property. Then $\mathcal{F}$ also has the mighty Erd\H{o}s-Hajnal property. 
\end{lemma}
\begin{proof}
Let $G \in \mathcal{F}$ and $A,B$ be disjoint vertex subsets of $G$ of the same size. The bipartite graph between $A$ and $B$ have edge density at least $1/2$ or at most $1/2$. In the first case, since $\mathcal{F}$ has the homogeneous density property, there is a constant $c>0$ and subsets $A' \subset A$ with $|A'| \geq c|A|$ and  $B' \subset B$ with $|B'| \geq c|B|$ such that $A'$ is complete to $B'$.  In the second case, since $\overline{\mathcal{F}}$ has the homogeneous density property, there is a constant $c>0$ and subsets $A' \subset A$ with $|A'| \geq c|A|$ and  $B' \subset B$ with $|B'| \geq c|B|$ such that $A'$ is empty to $B'$. So we get $\mathcal{F}$ has the mighty Erd\H{o}s-Hajnal property.   
\end{proof}

\begin{lemma}\label{mightytodensity}
Suppose that a hereditary family $\mathcal{F}$ has the mighty Erd\H{o}s-Hajnal property. Then $\mathcal{F}$ and $\overline{\mathcal{F}}$ both have the homogeneous density property. 
\end{lemma}

The proof of Lemma \ref{mightytodensity} is now a standard application of a weak bipartite version of Szemer\'edi's regularity lemma. The version we use is due to Koml\'os and S\'os (see \cite{KS96, pach98} or Theorem 9.4.1 on page 223 of \cite{mat}).

\begin{lemma}[\cite{KS96, pach98}]
\label{sz}
Let $G = (A,B,E)$ be a bipartite graph with parts $A$ and $B$ such that $|A| = |B| = n$.  Let $0< c \leq 1/2$.  If $|E| \geq \varepsilon n^2$, then there exists subsets $A_0 \subset A, B_0\subset B$ such that

\begin{enumerate}

\item  $|A_0| = |B_0| = \varepsilon^{1/c^2} n$, and

\item $|E(A_0,B_0)| \geq \varepsilon |A_0||B_0|$, and

\item $|E(A',B')| > 0$ for any $A'\subset A_0$ and $B' \subset B_0$  with $|A'|,|B'|\geq c|A_0|$.

\end{enumerate}

\end{lemma}

\begin{proof}[Proof of Lemma \ref{mightytodensity}]
Since $\mathcal{F}$ has the mighty Erd\H{o}s-Hajnal property, there is a constant $0 < c \leq 1/2$ such that for every $G \in \mathcal{F}$ and disjoint vertex subsets $A,B$ of $G$ of the same size, there are $A' \subset A$ and $B' \subset B$ with $|A'|,|B'| \geq c|A|$ such that the bipartite graph between $A'$ and $B'$ is either complete or empty. 

By symmetry, it suffices to show that $\mathcal{F}$ has the homogeneous density property.  Let $G \in \mathcal{F}$ and $A,B$ be disjoint vertex subsets of equal size such that the edge density between $A$ and $B$ is at least $\varepsilon$. By Lemma \ref{sz}, there are $A_0 \subset A$ and $B_0 \subset B$ such that $|A_0|=|B_0| = \varepsilon^{1/c^2}|A|$, the edge density between $A_0$ and $B_0$ is at least $\varepsilon$, and $|E(A',B')|>0$ for any $A' \subset A_0$ and $B' \subset B_0$ with $|A'|,|B'| \geq c|A_0|$. Applying the mighty Erd\H{o}s-Hajnal property to $A_0$ and $B_0$, we get subsets $A'$ and $B'$ with $|A'|,|B'| \geq c|A_0|$ that are complete or empty to each other. However, the above condition that there is at least one edge between $A'$ and $B'$ implies that $A'$ and $B'$ are complete to each other, completing the proof.  \end{proof}

\subsection{Deductions for intersection graphs of pseudo-segments}

By Theorem \ref{equivalences}, the main result in this paper (Theorem \ref{regularity}) is equivalent to saying that the family of intersection graphs of pseudo-segments has the mighty Erd\H{o}s-Hajnal property. Therefore, to prove Theorem \ref{regularity}, it suffices to prove the following result. 

\begin{theorem}\label{main}
 Let $\mathcal{R}$ be a set of $n$ red curves, and $\mathcal{B}$ be a set of $n$ blue curves in the plane such that $\mathcal{R}\cup \mathcal{B}$ is a collection of pseudo-segments.  
 
 Then there are subsets $\mathcal{R}'\subset \mathcal{R}$ and $\mathcal{B}'\subset \mathcal{B}$, where $|\mathcal{R}'|,|\mathcal{B}'| \geq \Omega(n)$, such that either every curve in $\mathcal{R}'$ crosses every curves in $\mathcal{B}'$, or every curve in $\mathcal{R}'$ is disjoint from every curves in $\mathcal{B}'$.
\end{theorem}

As mentioned earlier, a weaker result, that the family of intersection graphs of pseudo-segments has the strong Erd\H{o}s-Hajnal property, was proved earlier by Fox, Pach, and T\'oth \cite{fpt}.  We repeat it below for convenience. 

\begin{theorem} [\cite{fpt}]\label{ehfpt}
Let $\mathcal{C}$ be a collection on $n$ pseudo-segments in the plane.  Then there are subsets $\mathcal{C}_1,\mathcal{C}_2\subset \mathcal{C}$, where $|\mathcal{C}_1|,|\mathcal{C}_2|\geq \Omega(n)$, such that
    either every curve in $\mathcal{C}_1$ crosses every curve in $\mathcal{C}_2$, or every curve in $\mathcal{C}_1$ is disjoint from every curve in $\mathcal{C}_2$
\end{theorem}

By Theorem \ref{equivalences}, Theorem \ref{main} is equivalent to the following two theorems holding.  

\begin{theorem}\label{notold}
    There is an absolute constant $c > 0$ such that the following holds.  Let $\mathcal{R}$ be a collection of $n$ red curves, and $\mathcal{B}$ be a collection of $n$ blue curves in the plane such that $\mathcal{R}\cup\mathcal{B}$ is a collection of pseudo-segments. 

 If there are at least $\varepsilon n^2$ crossing pairs in $\mathcal{R}\times \mathcal{B}$, then there are subsets $\mathcal{R}'\subset \mathcal{R},\mathcal{B}' \subset \mathcal{B}$, where $|\mathcal{R}'|,|\mathcal{B}'| \geq \varepsilon^{c}n$, such that every curve in $\mathcal{R}'$ crosses every curve in $\mathcal{B}'$. 
\end{theorem}

\begin{theorem}\label{appdis}
    There is an absolute constant $c > 0$ such that the following holds.  Let $\mathcal{R}$ be a collection of $n$ red curves, and $\mathcal{B}$ be a collection of $n$ blue curves in the plane such that $\mathcal{R}\cup\mathcal{B}$ is a collection of pseudo-segments. 

 If there are at least $\varepsilon n^2$ disjoint pairs in $\mathcal{R}\times \mathcal{B}$, then there are subsets $\mathcal{R}'\subset \mathcal{R},\mathcal{B}' \subset \mathcal{B}$, where $|\mathcal{R}'|,|\mathcal{B}'| \geq \varepsilon^{c}n$, such that every curve in $\mathcal{R}'$ is disjoint from every curve in $\mathcal{B}'$. 
\end{theorem}

  Theorem \ref{appdis} is the first density-type result for \emph{disjointness graphs} of pseudo-segments, and is the key ingredient in the proof of Theorem \ref{disjoint}.  For the proof of Theorem \ref{main}, we need the following non-bipartite version of Theorem \ref{notold} established by the first two authors in \cite{fp2}.
  
  \begin{theorem}[\cite{fp2}]\label{old}
       There is an absolute constant $c'>0$ such that the following holds.  Let $\mathcal{C}$ be a collection of $n$ pseudo-segments in the plane with at least $\varepsilon n^2$ crossing pairs.  Then there are subsets $\mathcal{C}_1, \mathcal{C}_2  \subset \mathcal{C}$, each of size $c'\varepsilon n$, such that every curve in $\mathcal{C}_1$ crosses every curve in $\mathcal{C}_2$. 
  \end{theorem}

\subsection{The strong Erd\H{o}s-Hajnal property implies the polynomial R\"odl property} \label{strongeh}

The goal of this subsection is to prove the following lemma, and deduce that the family of intersection graphs of pseudo-segments has the polynomial R\"odl property. 

\begin{lemma}\label{strongEHimpliesPR}
If a hereditary family $\mathcal{F}$ of graphs has the strong Erd\H{o}s-Hajnal property, then it also has the polynomial R\"odl property.
\end{lemma}

In proving Lemma \ref{strongEHimpliesPR}, given a graph $G$ in a hereditary family of graphs satisfying the strong Erd\H{o}s-Hajnal property, we find a large subgraph of $G$ or its complement that is a balanced complete multipartite graph with many parts. Such a subgraph must have density close to one, giving the desired $\varepsilon$-homogeneous induced subgraph. 

\begin{proof}[Proof of Lemma \ref{strongEHimpliesPR}]
Since $\mathcal{F}$ has the strong Erd\H{o}s-Hajnal property, there is a constant $c>0$ such that every $G \in \mathcal{F}$ on $n \geq 2$ vertices has disjoint vertex subsets $V_1,V_2$ with $|V_1|=|V_2| = cn$ such that the graph between $V_1$ and $V_2$ is complete or empty. We repeatedly apply this to the induced subgraph of each of the vertex subsets of $G$ we obtain. After $t$ levels of iteration, we obtain $2^t$ disjoint vertex subsets, each of size $c^t n$, that are complete or empty between each pair of them, and the induced subgraph $H$ of $G$ formed by taking one vertex from each of these $2^t$ subsets is a cograph. Since cographs are perfect graphs, and any perfect graph with $m$ vertices contains a clique or independent set of size at least $m^{1/2}$, $H$ contains a clique or independent set with at least $2^{t/2}$ vertices. Let $U_1,\ldots,U_s$ be the vertex subsets that the vertices from this clique or independent set belong to. Then the edge density of the induced subgraph on the union $\bigcup_{i=1}^s U_i$ is either more than  $1-1/s$ or less than $1/s$. Let $t=2\log_2 (1/\varepsilon)$. Since $s \geq 2^{t/2} = 1/\varepsilon$ and $c^t=\varepsilon^C$ with $C=2\log_2 (1/c)$, this completes the proof.
\end{proof}

Given a collection $\mathcal{C}$ of curves in the plane,  let $G(\mathcal{C})$ denote the intersection graph of $\mathcal{C}$. According to Theorem~\ref{ehfpt}, the family of intersection graphs of pseudo-segments has the strong Erd\H{o}s-Hajnal property. Applying Lemma \ref{strongEHimpliesPR}, we obtain  

\begin{corollary}\label{lemhom}
The family of intersection graphs of pseudo-segments has the polynomial R\"odl property. That is, there is an absolute constant $c_1 > 0$ such that the following holds. Let $\varepsilon > 0$ and $\mathcal{C}$ be a collection of $n$ pseudo-segments in the plane.  Then there is a subset $\mathcal{C}'\subset \mathcal{C}$ of size $\varepsilon^{c_1} n$ whose  intersection graph $G(\mathcal{C}')$ is $\varepsilon$-homogeneous.  
\end{corollary}

We will frequently use the following simple lemma in this paper.  

\begin{lemma}\label{subhom}
    Let $G = (V,E)$ be a graph on $n$ vertices.  If the edge density of $G$ is at most $\varepsilon$, then any induced subgraph on $\delta n$ vertices has edge density at most $2\varepsilon/\delta^2$.  Likewise, if  the edge density of $G$ is at least $1 - \varepsilon$, then any induced subgraph on $\delta n$ vertices has edge density at least $1 - 2\varepsilon/\delta^2$.
\end{lemma}

\begin{proof}
Suppose the edge density of $G$ is at most $\varepsilon n^2$ edges.  Hence, any induced subgraph on $\delta n$ vertices will have at most $(2\varepsilon/\delta^2)(\delta n)^2/2$ edges, which implies that the edge density of the induced subgraph is at most $2\varepsilon/\delta^2$.  A symmetric argument follows in the case when the edge density of $G$ is at least $1-\varepsilon$.  \end{proof}

\section{Proof of Theorem \ref{main} -- for double grounded red curves}\label{dbground}

Given a collection of curves $\mathcal{C}$ in the plane, we say that $\mathcal{C}$ is \emph{double grounded} if there are two distinct curves $\gamma_1$ and $\gamma_2$ such that for each curve $\alpha \in \mathcal{C}$, $\alpha$ has one endpoint on $\gamma_1$ and the other on $\gamma_2$, and the interior $\alpha$ is disjoint from $\gamma_1$ and $\gamma_2$.  Throughout this paper, for simplicity, we will always assume that both endpoints of each of our curves have distinct $x$-coordinates.  We refer to the endpoint of a curve with the smaller (larger) $x$-coordinate as its \emph{left (right) endpoint}.    The aim of this section is to prove Theorem \ref{main} in the special case where one of the color classes (the red one, say) consists of double grounded curves.
 
A curve in the plane is called \emph{$x$-monotone} if every vertical line intersects it in at most one point.  We start by considering double grounded $x$-monotone curves, and at the end of this section, we will remove the $x$-monotone condition.  We will need the following result, known as the cutting-lemma for $x$-monotone curves. See, for example, Proposition 2.11 in \cite{hp}.

 \begin{lemma}[The Cutting Lemma]\label{xcut}

Let $\mathcal{C}$ be a collection of $n$ double grounded $x$-monotone curves, whose grounds that are disjoint vertical segments $\gamma_1$ and $\gamma_2$, and let $r > 1$ be a parameter.  Then $\mathbb{R}^2\setminus (\gamma_1\cup \gamma_2)$ can be subdivided into $t$ connected regions $\Delta_1,\ldots, \Delta_t$, such that the interior of each $\Delta_i$ is intersected by at most $n/r$ curves from $\mathcal{C}$, and we have $t  = O(r^2)$.        
 \end{lemma}

In the proof of the following lemma and throughout the paper, we will implicitly use the Jordan curve theorem.

\begin{lemma}\label{xmono}
 Let $\mathcal{R}$ be a set of $n$ red double grounded $x$-monotone curves, whose grounds are disjoint vertical segments $\gamma_1$ and $\gamma_2$.  Let $\mathcal{B}$ be a set of $n$ blue curves (not necessarily $x$-monotone) such that every blue curve in $\mathcal{B}$ is disjoint from grounds $\gamma_1$ and $\gamma_2$, and suppose that $\mathcal{R}\cup \mathcal{B}$ is a collection of pseudo-segments.  
 
 Then  then there are subsets $\mathcal{R}'\subset \mathcal{R}$ and $\mathcal{B}'\subset \mathcal{B}$ such that $|\mathcal{R}'|,|\mathcal{B}'| \geq \Omega(n)$, and either every curve in $\mathcal{R'}$ crosses every curve in $\mathcal{B}'$, or every curve in $\mathcal{R}'$ is disjoint from every curve in $\mathcal{B}'$.

\end{lemma}

\begin{proof}   

Let $P$ be the set of left-endpoints of the curves in $\mathcal{B}$.  We apply Lemma \ref{xcut} to $\mathcal{R}$ with parameter $r=4$ to obtain a subdivision

$$\mathbb{R}^2\setminus (\gamma_1\cup\gamma_2) = \Delta_1\cup \cdots \cup \Delta_t,$$

\noindent such that for each $\Delta_i$, the interior of $\Delta_i$ intersects at most $n/4$ members in $\mathcal{R}$, and $t \leq c_04^2$ where $c_0$ is an absolute constant from Lemma \ref{xcut}.  By the pigeonhole principle, there is a region $\Delta_i$ such that $\Delta_i$ contains at least $n/c_04^2$ points from $P$.   Let $\mathcal{B}_0 \subset \mathcal{B}$ be the set of blue curves whose left endpoints are in $\Delta_i$.  Hence $|\mathcal{B}_0| = \Omega(n)$.

Let $Q$ be the right endpoints of the curves in $\mathcal{B}_0$.  Using the same subdivision described above, there is a region $\Delta_j$ such that $\Delta_j$ contains at least $|Q|/(c_04^2) \geq n/(c_04^2)^2$ points from $Q$.  Let $\mathcal{B}_1\subset \mathcal{B}_0$ be the set of blue curves with their left endpoint in $\Delta_i$ and right endpoint in $\Delta_j$.  Let $\mathcal{R}_1\subset \mathcal{R}$ consists of all red curves that do not intersect the interior of $\Delta_i$ and $\Delta_j$.   Lemma \ref{xcut} implies that

$$|\mathcal{R}_1| \geq n - \frac{2n}{4} = \frac{n}{2},$$

    \noindent and $|\mathcal{B}_1| = \Omega(n).$ Recall that each blue curve in $\mathcal{B}_1$ does not intersect the grounds $\gamma_1$ nor $\gamma_2$.  Fix an arbitrary curve $\alpha_0 \in \mathcal{R}_1$.  The proof now falls into the following cases.

\medskip

\noindent \emph{Case 1.}  Suppose at least $|\mathcal{R}_1|/2$ curves in $\mathcal{R}_1$ are disjoint from $\alpha_0$.  Let $\mathcal{R}_2 \subset \mathcal{R}_1$ be the set of red curves disjoint from $\alpha_0$.  For each $\alpha \in \mathcal{R}_2\setminus \{\alpha_0\}$, 

$$\mathbb{R}^2\setminus (\gamma_1\cup\gamma_2\cup \alpha_0\cup \alpha),$$

\noindent consists of two connected components, one bounded and the other unbounded.

 \begin{figure}
  \centering
\subfigure[Case 1.a.]{\label{figc1a}\includegraphics[width=0.25\textwidth]{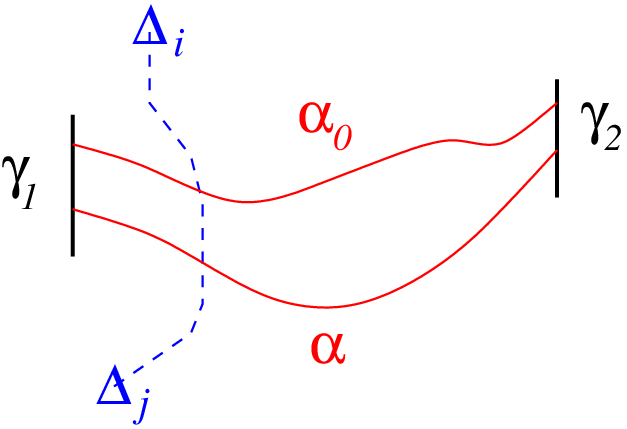}}  \hspace{2cm}
  \subfigure[Case 1.b.]{\label{figc1b}\includegraphics[width=0.25\textwidth]{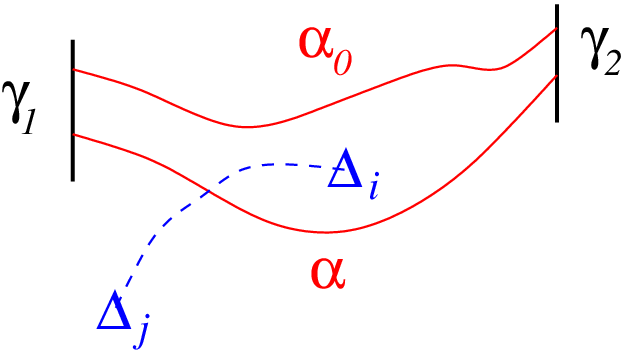}}
 \caption{Case 1, $\alpha_0$ and $\alpha$ are disjoint.}\label{figcase1}
\end{figure}

\medskip

\noindent \emph{Case 1.a.}  Suppose for at least $|\mathcal{R}_2|/2$ red curves  $\alpha \in \mathcal{R}_2$,  both $\Delta_i$ and $\Delta_j$ lie in the same connected component of $\mathbb{R}^2\setminus (\gamma_1\cup\gamma_2\cup \alpha_0\cup \alpha)$.  See Figure \ref{figc1a}. Let $\mathcal{R}_3\subset \mathcal{R}_2$ be the collection of such red curves.  Then for each $\alpha \in \mathcal{R}_3$, each blue curve $\beta \in \mathcal{B}_1$ crosses $\alpha$ if and only if $\beta$ crosses $\alpha_0$.  Hence, there is a subset $\mathcal{B}_2 \subset \mathcal{B}_1$ of size at least $\Omega(n)$, such that either every blue curve in $\mathcal{B}_2$ crosses every red curve in $\mathcal{R}_3$, or every blue curve in $\mathcal{B}_2$ is disjoint from every red curve in $\mathcal{R}_3$.  Moreover, $|\mathcal{R}_3| = \Omega(n)$ and we are done.

\medskip

\noindent \emph{Case 1.b.}  Suppose for at least $|\mathcal{R}_2|/2$ red curves  $\alpha \in \mathcal{R}_2$,  regions $\Delta_i$ and $\Delta_j$ lie in different connected component of  $\mathbb{R}^2\setminus (\gamma_1\cup\gamma_2\cup \alpha_0\cup \alpha)$.  See Figure \ref{figc1b}.  Similar to above, let $\mathcal{R}_3\subset \mathcal{R}_2$ be the collection of such red curves.  By the pseudo-segment condition, for each $\alpha \in \mathcal{R}_3$, each blue curve $\beta \in \mathcal{B}_1$ crosses $\alpha$ if and only if $\beta$ is disjoint from $\alpha_0$.  Hence, there is a subset $\mathcal{B}_2 \subset \mathcal{B}_1$ of size $\Omega_r(n)$, such that either every blue curve in $\mathcal{B}_2$ crosses every red curve in $\mathcal{R}_3$, or every blue curve in $\mathcal{B}_2$ is disjoint from every red curve in $\mathcal{R}_3$.  Moreover, $|\mathcal{R}_3| = \Omega(n)$ and we are done.

\medskip

\noindent \emph{Case 2.}  Suppose at least $|\mathcal{R}_1|/2$ curves in $\mathcal{R}_1$ cross $\alpha_0$.  Let $\mathcal{R}_2 \subset \mathcal{R}_1$ be the set of red curves that crosses $\alpha_0$.  For each $\alpha \in \mathcal{R}_2\setminus \{\alpha_0\}$, 

$$\mathbb{R}^2\setminus (\gamma_1\cup\gamma_2\cup \alpha_0\cup \alpha),$$

\noindent consists of three connected components, two of which are bounded and the other unbounded.  

\medskip

\noindent \emph{Case 2.a.}  Suppose for at least $|\mathcal{R}_2|/3$ red curves  $\alpha \in \mathcal{R}_2$,  Both $\Delta_i$ and $\Delta_j$ lie in the same connected component of  $\mathbb{R}^2\setminus (\gamma_1\cup\gamma_2\cup \alpha_0\cup \alpha)$.  See Figure \ref{figc2a}. Let $\mathcal{R}_3\subset \mathcal{R}_2$ be the collection of such red curves.  By the pseudo-segment condition, for each $\alpha \in \mathcal{R}_3$, each blue curve $\beta \in \mathcal{B}_1$ crosses $\alpha$ if and only if $\beta$ crosses $\alpha_0$.  Hence, there is a subset $\mathcal{B}_2 \subset \mathcal{B}_1$ of size at least $\Omega(n)$, such that either every blue curve in $\mathcal{B}_2$ crosses every red curve in $\mathcal{R}_3$, or every blue curve in $\mathcal{B}_2$ is disjoint from every red curve in $\mathcal{R}_3$.  Moreover, $|\mathcal{R}_3| = \Omega(n)$.

\medskip

\noindent \emph{Case 2.b.}  Suppose for at least $|\mathcal{R}_2|/3$ red curves  $\alpha \in \mathcal{R}_2$,  regions $\Delta_i$ and $\Delta_j$ lie in different bounded connected components of  $\mathbb{R}^2\setminus (\gamma_1\cup\gamma_2\cup \alpha_0\cup \alpha)$.  See Figure \ref{figc2b}. Let $\mathcal{R}_3\subset \mathcal{R}_2$ be the collection of such red curves.  Then for each $\alpha \in \mathcal{R}_3$, every blue curve $\beta \in \mathcal{B}_1$ crosses $\alpha$. Since $|\mathcal{R}_3| = \Omega(n)$ and $|\mathcal{B}_1| = \Omega(n)$.

\medskip

\noindent \emph{Case 2.c.}  Suppose for at least $|\mathcal{R}_2|/3$ red curves  $\alpha \in \mathcal{R}_2$,  regions $\Delta_i$ and $\Delta_j$ lie in different connected components of $\mathbb{R}^2\setminus (\gamma_1\cup\gamma_2\cup \alpha_0\cup \alpha)$, one of which is bounded and the other unbounded.  See Figure~\ref{figc2c}.  Let $\mathcal{R}_3\subset \mathcal{R}_2$ be the collection of such red curves.  By the pseudo-segment condition, for each $\alpha \in \mathcal{R}_3$, each blue curve $\beta \in \mathcal{B}_1$ crosses $\alpha$ if and only if $\beta$ is disjoint from $\alpha_0$.  Hence, there is a subset $\mathcal{B}_2 \subset \mathcal{B}_1$ of size $\Omega(n)$, such that either every blue curve in $\mathcal{B}_2$ crosses every red curve in $\mathcal{R}_3$, or every blue curve in $\mathcal{B}_2$ is disjoint from every red curve in $\mathcal{R}_3$.  Moreover, $|\mathcal{R}_3|  = \Omega(n)$, and we are done.  \end{proof}

 \begin{figure}
  \centering
\subfigure[Case 2.a.]{\label{figc2a}\includegraphics[width=0.25\textwidth]{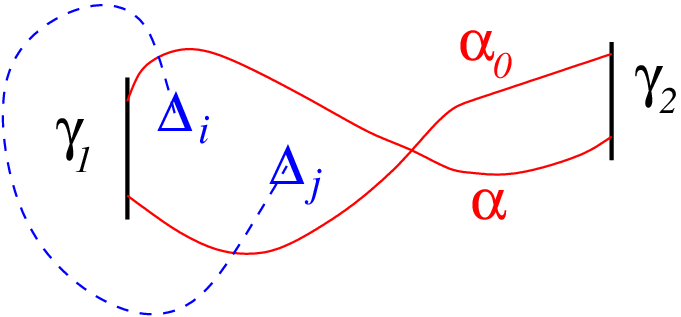}}  \hspace{1cm}
\subfigure[Case 2.b.]{\label{figc2b}\includegraphics[width=0.25\textwidth]{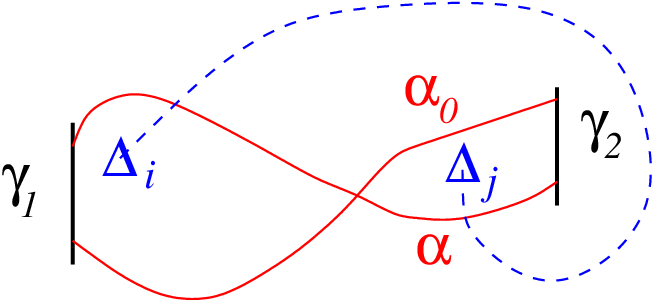}}  \hspace{1cm}
  \subfigure[Case 2.c.]{\label{figc2c}\includegraphics[width=0.25\textwidth]{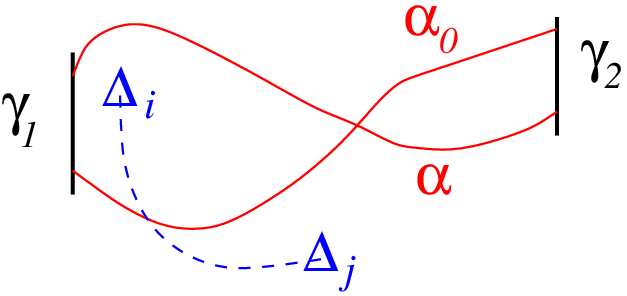}}
 \caption{Case 2, $\alpha_0$ and $\alpha$ cross.}\label{figcase2}
\end{figure}

Recall that a \emph{pseudoline} is an unbounded arc in $\mathbb{R}^2$, whose complement is disconnected.  An \emph{arrangement of pseudolines} is a set of pseudolines such that every pair meets exactly once, and no three members have a point in common.  A classic result of Goodman \cite{Go} states that every arrangement of pseudolines is isomorphic to an arrangement of \emph{wiring diagrams} (bi-infinite $x$-monotone curves).  Moreover, Goodman and Pollack showed the following.

 \begin{theorem}[\cite{GoP}]\label{GoP}
Every arrangement of pseudolines can be continuously deformed (through isomorphic arrangements) to a wiring diagram.     
 \end{theorem}

We also need the following simple lemma.

\begin{lemma}\label{sweep}
    Given a finite linearly ordered set whose elements are colored red or blue, we can select half of the red elements and half of the blue elements such that all of the selected elements of one color come before all of the selected elements of the other color.
\end{lemma}

We are now ready to establish the main result of this section:

\begin{theorem}\label{thmdg}
Let $\mathcal{R}$ be a set of $n$ red double grounded curves with grounds $\gamma_1$ and $\gamma_2$, where $\gamma_1$ and $\gamma_2$ cross each other.  Let $\mathcal{B}$ be a set of $n$ blue curves such that $\mathcal{R}\cup \mathcal{B}\cup \{\gamma_1, \gamma_2\}$ is a collection of pseudo-segments.   

Then there are subsets $\mathcal{R}'\subset \mathcal{R}$ and $\mathcal{B}'\subset \mathcal{B}$ such that $|\mathcal{R}'|,|\mathcal{B}'| \geq \Omega(n)$, and either every curve in $\mathcal{R'}$ crosses every curve in $\mathcal{B}'$, or every curve in $\mathcal{R}'$ is disjoint from every curve in $\mathcal{B}'$.
\end{theorem}

\begin{proof}
By passing to linear-sized subsets of $\mathcal{R}$ and $\mathcal{B}$ and subcurves of $\gamma_1$ and $\gamma_2$, we will reduce the problem to the setting of Lemma \ref{xmono}.  Let us assume that $\gamma_1$ and $\gamma_2$ cross at point $p$.  Hence, $(\gamma_1\setminus \gamma_2)\cup (\gamma_2\setminus \gamma_1)$ consists of four connected components.  By the pigeonhole principle, there is a subset $\mathcal{R}_1\subset \mathcal{R}$ of size $n/4$ such that every curve in $\mathcal{R}_1$ has an endpoint on one of the connected components of $\gamma_1\setminus \gamma_2$, and all of the other endpoints lie on one of the connected components of $\gamma_2\setminus \gamma_1$.  Let $\gamma'_i \subset\gamma_i$, for $i  = 1,2$, be these connected components so that they have a common endpoint at $p$ and their interiors are disjoint.

 For each $\alpha \in \mathcal{R}_1$, the sequence of curves $(\gamma'_1,\gamma'_2,\alpha)$ appear either in \emph{clockwise} or \emph{counterclockwise} order along the unique simple closed curve that lies in $\gamma'_1\cup\gamma'_2\cup\alpha$.  Without loss of generality, we can assume that there is a subset $\mathcal{R}_2 \subset \mathcal{R}_1$, where $|\mathcal{R}_2| = \Omega(n)$, such that for every curve  $\alpha \in \mathcal{R}_2$, the sequence $(\gamma'_1,\gamma'_2,\alpha)$ appears in clockwise order, since a symmetric argument would follow otherwise.

We define the \emph{orientation} of each curve $\alpha \in \mathcal{R}_2$ as the sequence of turns, either \emph{left-left}, \emph{left-right}, \emph{right-left}, or \emph{right-right}, made by starting at $p$ and moving along $\gamma'_1$ in the arrangement $\gamma_1'\cup\gamma_2'\cup \alpha$, until we return back to $p$.   More precisely, starting at $p$ we move along $\gamma'_1$ until we reach the endpoint of $\alpha$.  We then turn either \emph{left} or \emph{right} to move along $\alpha$ towards $\gamma'_2$.  Once we've reached $\gamma'_2$, we either turn \emph{left} or \emph{right} in order to move along $\gamma'_2$ and reach $p$ again.  By the pigeonhole principle, there is a subset $\mathcal{R}_3 \subset\mathcal{R}_2$ of size  at least $\Omega(n)$ such that all curves in $\mathcal{R}_3$ have the same orientation.   Without loss of generality, we can assume that the orientation is \emph{left-left}, since a symmetric argument would follow otherwise.

Starting at $p$ and moving along $\gamma_1'$ towards its other endpoint, let us consider the sequence of curves from $\mathcal{R}_3\cup\mathcal{B}$ intersecting $\gamma_1'$.  Then, by Lemma \ref{sweep}, there are subsets $\mathcal{R}_4\subset \mathcal{R}_3$ and $\mathcal{B}_1\subset \mathcal{B}$, where $|\mathcal{R}_4| \geq |\mathcal{R}_3|/2$ and $|\mathcal{B}_1| \geq |\mathcal{B}|/2$, such that either all of the curves in $\mathcal{R}_4$ appear before all of the curves in $\mathcal{B}_1$ that intersect $\gamma_1'$ in this sequence, or all of the curves in $\mathcal{R}_4$ appear after all of the curves in $\mathcal{B}_1$ in this sequence.  Note that $\mathcal{B}_1$ consists of the blue curves in $\mathcal{B}$ that are disjoint to $\gamma_1'$ and at least half of the curves in $\mathcal{B}$ that intersect $\gamma_1'$ found by the application of Lemma \ref{sweep}. Hence, there is a subcurve $\gamma_1''\subset\gamma'_1$ such that $\gamma_1''$ is one of the grounds for $\mathcal{R}_4$, and is disjoint from every curve in $\mathcal{B}_1$.  We apply the same argument to $\mathcal{R}_4\cup \mathcal{B}_1$ and $\gamma_2'$, and obtain subsets $\mathcal{R}_5\subset \mathcal{R}_4$, $\mathcal{B}_2\subset\mathcal{B}_1$, and a subcurve $\gamma_2'' \subset\gamma_2'$, such that $|\mathcal{R}_5|,|\mathcal{B}_2| = \Omega(n)$, and $\mathcal{R}_5$ is double grounded with disjoint grounds $\gamma_1''$ and $\gamma_2''$, and every curve in $\mathcal{B}_2$ is disjoint from $\gamma_1''$ and $\gamma_2''$.

For $i \in \{1,2\}$, let $p_i$ be the endpoint of $\gamma_i''$ that lies closest to $p$ along $\gamma'_i$.  Starting at $p_i$ and moving along $\gamma_i''$, let $\pi_i$ be the sequence of curves in $\mathcal{R}_5$ that appear on $\gamma_i''$.  Since every curve in $\mathcal{R}_5$ has the same left-left orientation, and appears clockwise order with respect to $\gamma_1'$ and $\gamma_2'$, two curves $\alpha,\alpha' \in \mathcal{R}_5$ cross if and only if the order in which they appear in $\pi_1$ and $\pi_2$ changes.  Let $\gamma''_3$ be a curve very close to $\gamma_2''$ such that $\gamma''_3$ has the same endpoints as $\gamma_2''$, and is disjoint from all curves in $\mathcal{R}_5\cup\mathcal{B}_2$.  Hence, 
$\gamma''_2\cup\gamma_3''$ makes an empty lens in the arrangement $\mathcal{R}_5\cup\mathcal{B}_2$.   We slightly extend each curve $\alpha \in \mathcal{R}_5$ through this lens to $\gamma_3''$ so that the resulting curve, $\alpha'$ properly crosses $\gamma''_2$ and has its new endpoint on $\gamma_3''$.  Moreover, the extension will be made in such a way that the sequence $\pi_3$ of curves in $\mathcal{R}_5$ appearing along $\gamma_3''$ starting from $p_2$ will appear in the opposite order of $\pi_1$.  Let 

$$\mathcal{R}_5' = \{\alpha': \alpha \in \mathcal{R}_5\}.$$

\noindent Thus,  every pair of curves in $\mathcal{R}_5'$ will cross exactly once.

For each curve $\alpha' \in \mathcal{R}'_5$, we further extend $\alpha'$ by moving both endpoints towards $p$ along $\gamma_1$ and $\gamma_2$, so that we do not create any additional crossings within $\mathcal{R}'_5$.  Let $\hat{\alpha}$ be the resulting extension, where both endpoints of $\hat{\alpha}$ lie arbitrarily close to $p$.  Set $\hat{\mathcal{R}}_5 = \{\hat{\alpha}: \alpha' \in \mathcal{R}'_5\}.$  See Figure \ref{figext}/ Furthermore, we can assume that $p$ lies in the unbounded face of the arrangement $\hat{\mathcal{R}}_5$, since otherwise we could project the arrangement $\hat{\mathcal{R}}_5$ onto a sphere, and then project it back to the plane so that $p$ lies in the unbounded face, without creating or removing any crossing.  Therefore, $\hat{\mathcal{R}}_5$ can be extended to a family of pseudolines.  By Theorem \ref{GoP}, we can apply a continuous deformation of the plane so that $\hat{\mathcal{R}}_5$ becomes a collection of unbounded $x$-monotone curves.   Hence, after the deformation, the original set $\mathcal{R}_5$ becomes a collection of double grounded $x$-monotone curves, with grounds $\gamma_1'',\gamma_2''$, such that every curve in $\mathcal{B}_2$ is disjoint from the grounds $\gamma_1''$ and $\gamma_2''$, the crossing pattern in the arrangement $\mathcal{R}_5\cup\mathcal{B}_2$ is the same as before.  Moreover, $\gamma_1''$ and $\gamma_2''$ will be disjoint vertical segments. We apply Lemma \ref{xmono} to $\mathcal{R}_5$ and $\mathcal{B}_2$ and obtain subsets $\mathcal{R}_6\subset \mathcal{R}_5$ and $\mathcal{B}_3 \subset \mathcal{B}_2$, each of size $\Omega(n)$, such that either every curve in $\mathcal{R}_6$ crosses every curve in $\mathcal{B}_3$, or every curve in $\mathcal{R}_6$ is disjoint from every curve in $\mathcal{B}_3$. This completes the proof.\end{proof}

\begin{figure}
\centering
\includegraphics[width=4.5cm]{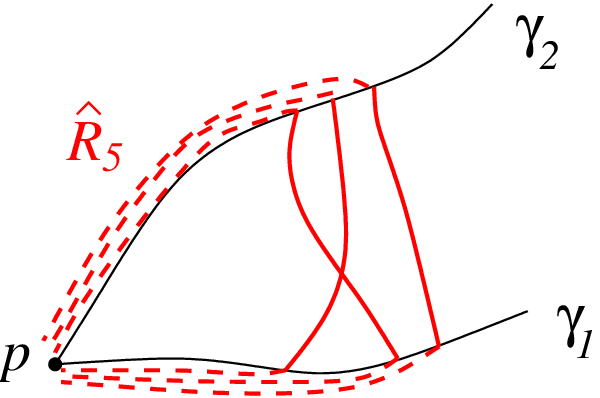}
\caption{The resulting extension $\hat{\mathcal{R}}_5$.}\label{figext}
\end{figure}

By combining Theorem \ref{thmdg} with Lemma \ref{sz}, as shown in Section \ref{her}, we obtain the following density theorems.

\begin{theorem}\label{dgold}
    There is an absolute constant $c' > 0$ such that the following holds.  Let $\mathcal{R}$ be a collection of $n$ red double grounded curves with grounds $\gamma_1$ and $\gamma_2$, such that $\gamma_1$ and $\gamma_2$ cross.  Let $\mathcal{B}$ be a collection of $n$ blue curves such that $\mathcal{R}\cup\mathcal{B}\cup \{\gamma_1, \gamma_2\}$ is a collection of pseudo-segments.
    
    If there are at least $\varepsilon n^2$ crossing pairs in $\mathcal{R}\times \mathcal{B}$, then there are subsets $\mathcal{R}'\subset \mathcal{R},\mathcal{B}' \subset \mathcal{B}$, where $|\mathcal{R}'|,|\mathcal{B}'| \geq \varepsilon^{c'}n$, such that every curve in $\mathcal{R}'$ crosses every curves in $\mathcal{B}'$. 
\end{theorem}

\begin{theorem}\label{dgappdis}
   There is an absolute constant $c' > 0$ such that the following holds.  Let $\mathcal{R}$ be a collection of $n$ red double grounded curves with grounds $\gamma_1$ and $\gamma_2$, such that $\gamma_1$ and $\gamma_2$ cross.  Let $\mathcal{B}$ be a collection of $n$ blue curves such that $\mathcal{R}\cup\mathcal{B}$ is a collection of pseudo-segments.

 If there are at least $\varepsilon n^2$ disjoint pairs in $\mathcal{R}\times \mathcal{B}$, then there are subsets $\mathcal{R}'\subset \mathcal{R},\mathcal{B}' \subset \mathcal{B}$, where $|\mathcal{R}'|,|\mathcal{B}'| \geq \varepsilon^{c'}n$, such that every curve in $\mathcal{R}'$ is disjoint from every curves in $\mathcal{B}'$. 
\end{theorem}

\noindent We will apply Theorems \ref{thmdg}  and \ref{dgold}  in the next section.

  \section{Proof of Theorem \ref{main} -- for $\varepsilon$-homogeneous families}\label{highvslow}

The aim of this section is to prove Theorem~\ref{main}, the main result of this paper, in the special case where the edge density of the intersection graph of the red curves is nearly 0 or nearly 1, and the same is true for the intersection graph of the blue curves.  This will easily imply Theorem~\ref{main} in its full generality, as shown in the next section.

 \subsection{Low versus low density}

  By Corollary \ref{lemhom}, it is sufficient to establish Theorem~\ref{main} in the special case where the intersection graphs $G(\mathcal{R})$ and $G(\mathcal{B})$ are both $\varepsilon$-homogeneous.  Below, we first consider the cases when the edge densities of both $G(\mathcal{R}) nd G(\mathcal{B})$ are smaller than $\varepsilon.$  Recall that a \emph{separator} for a graph $G = (V, E)$ is a subset $V_0 \subset V$ such that there
is a partition $V = V_1\cup V_1 \cup V_2$ with $|V_1|, |V_2| \leq \frac{2}{3}|V|$, and no vertex in $V_1$ is
adjacent to any vertex in $V_2$.   The celebrated Lipton-Tarjan separator theorem states
that every planar graph with $n$ vertices has a separator of size $O(\sqrt{n})$.  We will need the following result due to the first two authors.  For a strengthening, see \cite{lee}.

\begin{lemma}[\cite{fp}]\label{foxpach}

Let $\mathcal{C}$ be a collection of curves in the plane with a total of $m$ crossings. 
 Then the intersection graph $G(\mathcal{C})$ has a separator of size $O(\sqrt{m})$.
    
\end{lemma}

 We now prove the following.

\begin{lemma}\label{lemsep}
There is an absolute constant $\varepsilon_0 > 0$ such that the following holds.
  Let $\mathcal{R}$ be a set of $n$ red curves in the plane and let $\mathcal{B}$ be a set of $n$ blue curves such that $\mathcal{R}\cup \mathcal{B}$ is a collection of pseudo-segments.  
  
  If the edge densities of the intersection graphs $G(\mathcal{R})$ and $G(\mathcal{B})$ are both less than $\varepsilon_0$, and there are at most $\varepsilon_0 n^2$ crossing pairs in $\mathcal{R}\times \mathcal{B}$, then there are subsets $\mathcal{R}'\subset\mathcal{R}$ and $\mathcal{B}' \subset \mathcal{B}$, each of size $\Omega(n)$, such that every red curve in $\mathcal{R}'$ is disjoint from every blue curve in $\mathcal{B}'$.

\end{lemma}

\begin{proof}
Let $\varepsilon_0$ be a small constant that will be determined later. Set $\mathcal{C} = \mathcal{R}\cup\mathcal{B}$.  We apply Lemma~\ref{foxpach} to $\mathcal{C}$ and obtain a partition $\mathcal{C} = \mathcal{C}_0\cup \mathcal{C}_1\cup \mathcal{C}_2$, such that $|\mathcal{C}_1|,|\mathcal{C}_2|\leq (2/3)(2n)$, and every curve in $\mathcal{C}_1$ is disjoint from every curve in $\mathcal{C}_2$.  Let $c$ be the constant from Lemma \ref{foxpach}.  By setting $\varepsilon_0 < \frac{1}{(2c)^216}$, we have

$$|\mathcal{C}_0| \leq c\sqrt{\varepsilon_0(2n)^2} \leq 2c\sqrt{\varepsilon_0}n \leq n/4.$$

Without loss of generality, we can assume that at least half of the curves in $\mathcal{C}_1$ are red, since a symmetric argument would follow otherwise.  
Since $|\mathcal{C}_1| \geq (2n)/3 - |\mathcal{C}_0|$, there are at least $$\frac{2n/3 - n/4}{2} = 5n/24$$ red curves in $\mathcal{C}_1$.  Moreover, there are at most $2n/3$ blue curves in $\mathcal{C}_1$, which implies that there are at least

$$n - \frac{2n}{3} - \frac{n}{4} \geq \frac{n}{12},$$

\noindent blue curves in $\mathcal{C}_2$.  By setting $\mathcal{R}'$ to be the red curves in $\mathcal{C}_1$ and $\mathcal{B}'$ to be the blue curves in $\mathcal{C}_2$, we obtain the desired subsets. \end{proof}

\begin{theorem}\label{0a0}
There is an absolute constant $\varepsilon_1>0$ such that the following holds.  Let $\mathcal{R}$ be a set of $n$ red curves and $\mathcal{B}$ be a set of $n$ blue curves in the plane such that $\mathcal{R}\cup \mathcal{B}$ is a collection of pseudo-segments.  

If the edge densities of the intersection graphs $G(\mathcal{R})$ and $G(\mathcal{B})$ are both less than $\varepsilon_1$, then there are subsets $\mathcal{R}'\subset\mathcal{R}$ and $\mathcal{B}' \subset \mathcal{B}$, each of size $\Omega(n)$, such that every red curve in $\mathcal{R}'$ crosses every blue curve in $\mathcal{B}'$, or  every red curve in $\mathcal{R}'$ is disjoint from every blue curve in $\mathcal{B}'$.  
\end{theorem}

\begin{proof}
Let $\varepsilon_1$ be a small constant that will be determined later. Suppose there are at most $\varepsilon_0 n^2$ crossing pairs in $\mathcal{R}\times \mathcal{B}$, where $\varepsilon_0$ is defined in Lemma \ref{lemsep}. Since the edge density of the intersection graphs $G(\mathcal{R})$ and $G(\mathcal{B})$ are both less than $\varepsilon_1$, by setting $\varepsilon_1 < \varepsilon_0$, Lemma \ref{lemsep} implies that there are subsets $\mathcal{R}'\subset\mathcal{R}$ and $\mathcal{B}' \subset \mathcal{B}$, each of size $\Omega(n)$, such that every red curve in $\mathcal{R}'$ is disjoint from every blue curve in $\mathcal{B}'$ and we are done.  
    
    If there are at at least $\varepsilon_0 n^2$ crossing pairs in $\mathcal{R}\times \mathcal{B}$, then, by Theorem \ref{old}, there are subsets $\mathcal{C}_1,\mathcal{C}_2 \subset \mathcal{R}\cup\mathcal{B}$, each of size $\varepsilon_0^{c'}n$, where $c'>1$ is an absolute constant from Theorem \ref{old}, such that each curve in $\mathcal{C}_1$ crosses each curve in $\mathcal{C}_2$.  Without loss of generality, we can assume that at least $\varepsilon_0^{c'} n/2$ curves in $\mathcal{C}_1$ are red.  Since $G(\mathcal{R})$ has edge density at most $\varepsilon_1$, by setting $\varepsilon_1 = \varepsilon_0^{2c'}/4 <  \varepsilon_0$, at most $\varepsilon_0^{c'} n/2$ curves in $\mathcal{C}_2$ are red.  Hence, at least $\varepsilon_0^{c'} n/2$ curves in $\mathcal{C}_2$ are blue and we are done.\end{proof}

\subsection{High versus low edge density}

In this subsection, we consider the case when the intersection graph $G(\mathcal{R})$ has edge density at least $1-\varepsilon$, and $G(\mathcal{B})$ has edge density less than $\varepsilon$. Since the edge density in the intersection graph $G(\mathcal{R})$ is at least $1-\varepsilon$, we can further reduce to the case when there is a red curve $\gamma_1$ that crosses every member in $\mathcal{R}$ exactly once.

 \begin{lemma}\label{keylem}

For each integer $t \geq 1$, there is a constant $\varepsilon'_t > 0$ such that the following holds.  Let $\mathcal{R}$ be a set of $n$ red curves in the plane, all crossed by a curve $\gamma_1$ exactly once, and $\mathcal{B}$ be a set of $n$ blue curves in the plane such that $\mathcal{R}\cup \mathcal{B}\cup \{\gamma_1\}$ is a collection of pseudo-segments.  Suppose that the intersection graph $G(\mathcal{B})$ has edge density less than $\varepsilon'_t$, and $G(\mathcal{R})$ has edge density at least $1-\varepsilon'_t$. 

Then there are subsets $\hat{\mathcal{R}}\subset\mathcal{R}$, $\hat{\mathcal{B}}\subset \mathcal{B}$, each of size $\Omega_{\varepsilon'_t}(n)$, such that either every red curve in $\hat{\mathcal{R}}$ crosses every blue curve in $\hat{\mathcal{B}}$, or every red curve in $\hat{\mathcal{R}}$ is disjoint from every blue curve in $\hat{\mathcal{B}}$, or each curve $\alpha \in \hat{\mathcal{R}}$ has a partition into two connected parts $\alpha = \hat{\alpha}_{u}\cup \hat{\alpha}_{\ell}$, such that for 

 $$\hat{\mathcal{U}} =\{\hat{\alpha}_{u}: \alpha \in \hat{\mathcal{R}}, \alpha = \hat{\alpha}_{u}\cup \hat{\alpha}_{\ell}\}\hspace{.5cm}\textnormal{and}\hspace{.5cm}\hat{\mathcal{L}}=\{\hat{\alpha}_{\ell}: \alpha  \in \hat{\mathcal{R}}, \alpha = \hat{\alpha}_{u}\cup \hat{\alpha}_{\ell}\},$$

 \noindent every curve in $\hat{\mathcal{L}}$ is disjoint to every curve in $\hat{\mathcal{B}}$, and the edge density of $G(\hat{\mathcal{U}})$ is less that $2^{-t}$.

 \end{lemma}

\begin{proof}
 Each curve $\alpha \in \mathcal{R}$ is partitioned into two connected parts by $\gamma_1$, say an upper and lower part.  More precisely, we have the partition $\alpha = \alpha_{u}\cup \alpha_{\ell}$, where the parts $\alpha_{u}$ and $\alpha_{\ell}$ are defined, as follows. We start at the left endpoint of $\gamma_1$ and move along $\gamma_1$ until we reach $\alpha\cap \gamma_1$.  At this point, we turn left along $\alpha$ to obtain $\alpha_{u}$ and right  to obtain $\alpha_{\ell}$.  See Figure \ref{figupperlower}.  Let $\mathcal{U}$ ($\mathcal{L}$) be the upper (lower) part of each curve in $\mathcal{R}$, that is, 
 $$\mathcal{U} =\{\alpha_{u}: \alpha \in \mathcal{R}, \alpha  =  \alpha_{\ell}\cup \alpha_u\}\hspace{.5cm}\textnormal{and}\hspace{.5cm}\mathcal{L}=\{\alpha_{\ell}: \alpha  \in \mathcal{R}, \alpha =   \alpha_{\ell}\cup \alpha_u\}.$$

\begin{figure}
\centering
\includegraphics[width=4cm]{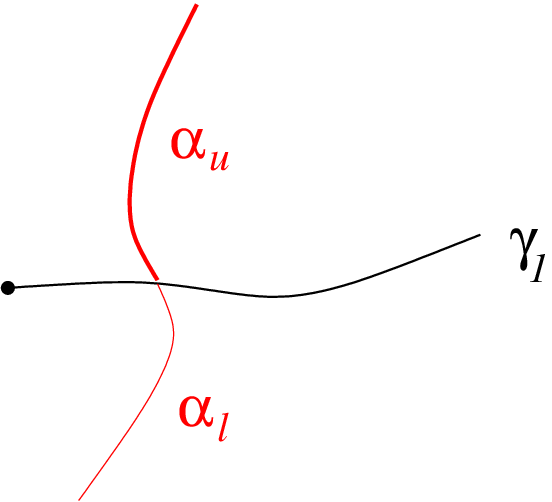}
\caption{Partitioning of the red curve $\alpha = \alpha_u\cup \alpha_{\ell}.$}\label{figupperlower}
\end{figure}

In what follows, for every integer $t\geq 1$, we will obtain subsets $\mathcal{R}^{(t)}\subset\mathcal{R}$, $\mathcal{B}^{(t)}\subset \mathcal{B}$, each of size $\Omega_{\varepsilon'_t}(n)$, such that either every red curve in $\mathcal{R}^{(t)}$ crosses every blue curve in $\mathcal{B}^{(t)}$, or every red curve in $\mathcal{R}^{(t)}$ is disjoint from every blue curve in $\mathcal{B}^{(t)}$, or each curve $\alpha \in \mathcal{R}^{(t)}$ has a new partition into two connected parts $\alpha = \alpha'_{u}\cup \alpha'_{\ell}$, upper and lower, such that the following holds.
     
     \begin{enumerate}

\item  We have $\alpha'_u \subset \alpha_u$, that is, the upper part $\alpha'_u$ is a subcurve of the previous upper part $\alpha_u$.  
     
\item The lower part $\alpha'_{\ell}$ of each curve in $\mathcal{R}^{(t)}$ is disjoint from each blue curve in $\mathcal{B}^{(t)}.$

     \item There is an equipartition $\mathcal{R}^{(t)}  = \mathcal{R}^{(t)}_{1}\cup \cdots \cup \mathcal{R}^{(t)}_{2^{t}}$ into $2^{t}$ parts such that for $1 \leq i < j \leq 2^{t-1}$, the upper part $\alpha'_u$ of each curve $\alpha \in \mathcal{R}^{(t)}_{i}$ is disjoint from the upper part $\beta'_u$ of each curve $\beta \in \mathcal{R}^{(t)}_{j}$.
     \end{enumerate}

\noindent Hence, the lemma follows from the statement above by setting $\hat{\mathcal{B}} = \mathcal{B}^{(t)}$, $\hat{\mathcal{R}} = \mathcal{R}^{(t)}$.

We proceed by induction on $t$.  The bulk of the argument below is actually for the base case $t = 1$, since we will just repeat the entire argument for the inductive step with parameter $\varepsilon'_t$.  Let $\varepsilon'_1$ be a small positive constant that will be determined later such that $\varepsilon'_1 < \varepsilon_1$, where $\varepsilon_1$ is from Theorem~\ref{0a0}. Thus, $G(\mathcal{R})$ has edge density at least $1 - \varepsilon'_1$ and $G(\mathcal{B})$ has edge density less than $\varepsilon'_1$. 

 Let $\delta > 0$ also be a sufficiently small constant determined later, such that $\varepsilon_1' < \delta < \varepsilon_1$. We apply Corollary \ref{lemhom} to $\mathcal{L}$ with parameter $\delta$ and obtain a subset $\mathcal{L}_1\subset \mathcal{L}$ such that $\mathcal{L}_1$ is $\delta$-homogeneous and $|\mathcal{L}_1| = \Omega_{\delta}(n)$.  Let $\mathcal{R}_1\subset \mathcal{R}$ be the red curves in $\mathcal{R}$ corresponding to the curves in 
 $\mathcal{L}_1$, and let $\mathcal{U}_1\subset \mathcal{U}$ be the curves in $\mathcal{U}$ that corresponds to the red curves in $\mathcal{R}_1$.

Without loss of generality, we can assume that the intersection graph $G(\mathcal{L}_1)$ has edge density less than $\delta$.  Indeed, otherwise if $G(\mathcal{L}_1)$ has edge density greater than $ 1- \delta$, by the pseudo-segment condition, the intersection graph $G(\mathcal{U}_1)$ must have edge density less than $\delta$ and a symmetric argument would follow.  In order to apply Theorem \ref{0a0}, we need two subsets of equal size.  By averaging, there is a subset $\mathcal{B}' \subset \mathcal{B}$ with $|\mathcal{B}'|=|\mathcal{L}_1|$ such that the edge density of $G(\mathcal{B}')$ is at most that of $G(\mathcal{B})$. Since $G(\mathcal{L}_1)$ has edge density less than $\delta$ and $G(\mathcal{B}')$ has edge density less than $\varepsilon'_1$, by setting $\varepsilon'_1 < \delta < \varepsilon_1$, we can apply Theorem \ref{0a0} to $\mathcal{L}_1$ and $\mathcal{B}'$ and obtain subsets $\mathcal{L}_2\subset \mathcal{L}_1$ and $\mathcal{B}_1\subset \mathcal{B}'$, each of size $\Omega_{\delta}(n)$, such that every curve in $\mathcal{L}_2$ crosses every blue curve in $\mathcal{B}_1$, or every curve in $\mathcal{L}_2$ is disjoint from every blue curve in $\mathcal{B}_1$.  If we are in the former case, then we are done.  Hence, we can assume that we are in the latter case.  Let $\mathcal{R}_2\subset \mathcal{R}_1$ be the red curves that corresponds to $\mathcal{L}_2$, and let $\mathcal{U}_2 \subset\mathcal{U}_1$ be the curves in $\mathcal{U}_1$ that corresponds to $\mathcal{R}_2$.  We apply Corollary~\ref{lemhom} to $\mathcal{U}_2$ with parameter $\delta$ and obtain a subset $\mathcal{U}_3 \subset \mathcal{U}_2$ such that $\mathcal{U}_3$ is $\delta$-homogeneous and $|\mathcal{U}_3| =  \Omega_{\delta}(n)$.  Let $\mathcal{R}_3$ be the red curves in $\mathcal{R}$ corresponding to $\mathcal{U}_3$, and let $\mathcal{L}_3$ be the curves in $\mathcal{L}_2$ that corresponds to $\mathcal{R}_3$.  

Suppose that the intersection graph $G(\mathcal{U}_3)$ has edge density less than $\delta$.  Since $|\mathcal{B}_1| = \delta_0n$, where $\delta_0 = \delta_0(\delta, \varepsilon_1)$, by Lemma~\ref{subhom}, the intersection graph $G(\mathcal{B}_1)$ has edge density at most $2\varepsilon'_1/\delta^2_0$. Thus, we set $\delta$ and $\varepsilon'_1$ sufficiently small so that $\delta < \varepsilon_1$ and $2\varepsilon'_1/\delta^2_0 
 <\varepsilon_1$.  By averaging, we can find subsets of $\mathcal{U}_3$ and $\mathcal{B}_1$, each of size $\min(|\mathcal{U}_3|,|\mathcal{B}_1|)$ and with densities less than $\varepsilon_1$, and apply Theorem \ref{0a0} to these subsets and obtain subsets $\mathcal{U}_4\subset \mathcal{U}_3$ and  $\mathcal{B}_2\subset \mathcal{B}_1$, each of size $\Omega_{\delta}(n),$ such that every curve in $\mathcal{U}_4$ crosses every blue curve in $\mathcal{B}_2$, or every curve in $\mathcal{U}_4$ is disjoint from every blue curve in $\mathcal{B}_2$.  In both cases, we are done since every curve in $\mathcal{L}_3$ is disjoint from every curve in $\mathcal{B}_2$.  Therefore, we can assume that $G(\mathcal{U}_3)$ has edge density greater than $1 -\delta$.  
   
For each curve $\alpha \in \mathcal{U}_3$, let $N(\alpha)$ denote the set of curves in $\mathcal{U}_3$ that intersects $\alpha$, and let $d(\alpha) = |N(\alpha)|$.  We label the curves $\beta\in N(\alpha)$ with integers $0$ to $d(\alpha)-1$ according to their closest intersection point to the ground $\gamma_1$ along $\alpha$, that is, the label $f_{\alpha}(\beta)$ of $\beta \in N(\alpha)$ is the number of curves in $\mathcal{U}_3$ that intersects the portion of $\alpha$ strictly between $\gamma_1$ and $\alpha\cap \beta$.  Since

$$\sum\limits_{\alpha \in \mathcal{U}_3} d(\alpha)-1  \geq 2(1-\delta)\binom{|\mathcal{U}_3|}{2}  - |\mathcal{U}_3|,$$

\noindent by Jensen's inequality, we have

$$\sum\limits_{\alpha \in \mathcal{U}_3}\sum\limits_{\beta\in N(\alpha)}f_{\alpha}(\beta) = \sum\limits_{\alpha \in \mathcal{U}_3}\binom{d(\alpha) }{2} \geq |\mathcal{U}_3| \binom{\frac{\sum_{\alpha \in \mathcal{U}_3}d(\alpha) }{|\mathcal{U}_3|}}{2} \geq \frac{|\mathcal{U}_3|^3}{  4}.$$

\noindent Let the weight $w(\beta)$ of a curve $\beta \in \mathcal{U}_3$ be the sum of its labels, that is,

$$w(\beta) = \sum\limits_{\alpha: \beta \in N(\alpha)}f_{\alpha}(\beta).$$

\noindent   Hence, the weight $w(\beta)$ is the total number of crossing points along curves $\alpha$ strictly between $\gamma_1$ and $\beta$, where $\alpha$ crosses both $\gamma_1$ and $\beta$. By averaging, there is a curve $\gamma_2 \in \mathcal{U}_3$ whose weight is at least $|\mathcal{U}_3|^2/4$.  

Using $\gamma_2$, we partition each curve $\alpha \in \mathcal{U}_3\setminus\{\gamma_2\}$ that crosses $\gamma_2$ into two connected parts, $\alpha = \alpha_w\cup \alpha_m$, where $\alpha_m$ is the connected subcurve with endpoints on $\gamma_1$ and $\gamma_2$, and $\alpha_w$ is the other connected part.  Set

$$\mathcal{W}_3 = \{\alpha_w: \alpha \in \mathcal{U}_3\setminus \{\gamma_2\}, \alpha\cap \gamma_2 \neq \emptyset\}\hspace{.5cm}\textnormal{and}\hspace{.5cm}\mathcal{M}_3 = \{\alpha_m: \alpha \in \mathcal{U}_3\setminus \{\gamma_2\},\alpha\cap \gamma_2 \neq \emptyset\}.$$

Since $\gamma_2$ has weight at least $|\mathcal{U}_3|^2/4$, by the pigeonhole principle, there are at least $|\mathcal{U}_3|^2/8$ intersecting pairs in $\mathcal{M}_3\times \mathcal{M}_3$, or at least $|\mathcal{U}_3|^2/8$ intersecting pairs in $\mathcal{M}_3\times \mathcal{W}_3$.

\medskip

\noindent \emph{Case 1.}  Suppose there are at least $|\mathcal{U}_3|^2/8$ pairs in $\mathcal{M}_3\times \mathcal{W}_3$ that cross. The set $\mathcal{M}_3$ is double grounded with grounds $\gamma_1$ and $\gamma_2$ that cross exactly once, and every curve in $\mathcal{W}_3$ is disjoint from $\gamma_1$ and $\gamma_2$. As $|\mathcal{M}_3|,|\mathcal{W}_3| \leq |\mathcal{U}_3|$, the density of edges in the bipartite intersection graph of $\mathcal{M}_3$ and $\mathcal{W}_3$ is at least $1/8$. By averaging, we can find subsets of $\mathcal{M}_3$ and $\mathcal{W}_3$ 
each of size $\min(|\mathcal{M}_3|,|\mathcal{W}_3|)$ such that the density of edges in the bipartite intersection graph of these subsets is at least $1/8$. By setting $\delta>0$ sufficiently small, we can apply Theorem \ref{dgold} to these subsets of $\mathcal{M}_3$ and $\mathcal{W}_3$ and obtain subsets $\mathcal{M}_4\subset\mathcal{M}_3$ and $\mathcal{W}'_4\subset \mathcal{W}_3$, each of size $\Omega_{\delta}(n)$, such that each curve in $\mathcal{M}_4$ crosses each curve in $\mathcal{W}'_4$.  Moreover, by the pseudo-segment condition, each curve in $\mathcal{M}_4\cup \mathcal{W}'_4$ corresponds to a unique curve in $\mathcal{U}_3$.  Let $\mathcal{U}_4\subset \mathcal{U}_3$ be the curves that corresponds to $\mathcal{M}_4$ and let $\mathcal{U}_4'\subset \mathcal{U}_3$ be the curves that corresponds to $\mathcal{W}'_4$.  Hence, we set

$$\mathcal{W}_4 = \{\alpha_m: \alpha \in \mathcal{U}_4, \alpha = \alpha_w\cup\alpha_m\}\hspace{.5cm}\textnormal{and}\hspace{.5cm}\mathcal{M}'_4 = \{\alpha_m: \alpha \in \mathcal{U}'_4, \alpha = \alpha_w\cup\alpha_m\}.$$

\noindent See Figure \ref{figca1}.

We apply Theorem \ref{thmdg} to arbitrary subsets of $\mathcal{M}_4$ and $\mathcal{B}_2$, each of size $\min(|\mathcal{M}_4|,|\mathcal{B}_2|)$, and obtain subsets $\mathcal{M}_5 \subset \mathcal{M}_4$ and $\mathcal{B}_3\subset \mathcal{B}_2$, each of size $\Omega_{\delta}(n)$, such that either every red curve in $\mathcal{M}_5$ crosses every blue curve in $\mathcal{B}_3$, or every red curve in $\mathcal{M}_5$ is disjoint from every blue curve in $\mathcal{B}_3$.  In the former case, we are done.  Hence, we can assume that we are in the latter case.  

We again apply Theorem \ref{thmdg} to arbitrary subsets of  $\mathcal{M}'_4$ and $\mathcal{B}_3$, each of size $\min(|\mathcal{M}'_4|,|\mathcal{B}_3|)$,  to obtain subsets $\mathcal{M}'_5 \subset \mathcal{M}'_4$ and $\mathcal{B}_4\subset \mathcal{B}_3$, each of size $\Omega_{\delta}(n)$,  such that either every red curve in $\mathcal{M}'_5$ crosses every blue curve in $\mathcal{B}_4$, or every red curve in $\mathcal{M}'_5$ is disjoint from every blue curve in $\mathcal{B}_4$.  Again, if we are in the former case, we are done.  Hence, we can assume that we are in the latter case.   Let

$$\mathcal{W}_5 = \{\alpha_w: \alpha = \alpha_w\cup\alpha_m, \alpha_m \in \mathcal{M}_5\}\hspace{.5cm}\textnormal{and}\hspace{.5cm}\mathcal{W}'_5 = \{\alpha_w: \alpha = \alpha_w\cup\alpha_m, \alpha_m \in \mathcal{M}'_5\},$$

\noindent and recall that every element in $\mathcal{M}_5$ crosses every element in $\mathcal{W}'_5$.  By the pseudo-segment condition, every element in $\mathcal{W}_5$ is disjoint from every element in $\mathcal{W}'_5$.

Let $\mathcal{R}_5$ be the red curves in $\mathcal{R}$ that corresponds to $\mathcal{W}_5$, and let $\mathcal{R}'_5$ be the red curves in $\mathcal{R}$ that corresponds to $\mathcal{W}_5'$.  We have $|\mathcal{R}_5|, |\mathcal{R}'_5| = \Omega_{\delta}(n)$, and moreover, we can assume that $|\mathcal{R}_5| = |\mathcal{R}'_5|$.  For each curve $\alpha \in \mathcal{R}_5\cup \mathcal{R}'_5$, and its original partition $\alpha = \alpha_u\cup \alpha_{\ell}$ defined by $\gamma_1$, we have a new partition $\alpha = \alpha_u' \cup \alpha_{\ell}'$ defined by $\gamma_2$, where $\alpha_u' = \alpha_w$ and $\alpha'_{\ell} = \alpha_m \cup \alpha_{\ell}$.  By setting $\mathcal{R}^{(1)} = \mathcal{R}_5\cup \mathcal{R}'_5$, and $\mathcal{B}^{(1)}  = \mathcal{B}_4$, where each curve $\alpha \in \mathcal{R}^{(1)}$ is equipped with the partition  $\alpha = \alpha_u' \cup \alpha_{\ell}'$, we satisfy the base case of the statement.

\medskip

\noindent \emph{Case 2}.   The argument is essentially the same as Case 1.  Suppose we have at least  $|\mathcal{U}_3|^2/8$ crossing pairs in $\mathcal{M}_3\times \mathcal{M}_3$.  Then by Theorem \ref{ehfpt}, there are subsets $\mathcal{M}_4,\mathcal{M}_4' \subset \mathcal{M}_3$, each of size $\Omega_{\delta}(n)$, such that every curve in $\mathcal{M}_4$ crosses every curve in $\mathcal{M}_4'$.  Let $\mathcal{U}_4\subset \mathcal{U}$ be the curves that corresponds to $\mathcal{M}_4$ and let $\mathcal{U}_4'\subset \mathcal{U}$ be the curves that corresponds to $\mathcal{M}_4'$.  Set

$$\mathcal{W}_4 = \{\alpha_w: \alpha \in \mathcal{U}_4, \alpha = \alpha_w\cup\alpha_m\}\hspace{.5cm}\textnormal{and}\hspace{.5cm}\mathcal{W}'_4 = \{\alpha_w: \alpha \in \mathcal{U}'_4, \alpha = \alpha_w\cup\alpha_m\}.$$

 \begin{figure}
  \centering
\subfigure[Case 1.]{\label{figca1}\includegraphics[width=0.25\textwidth]{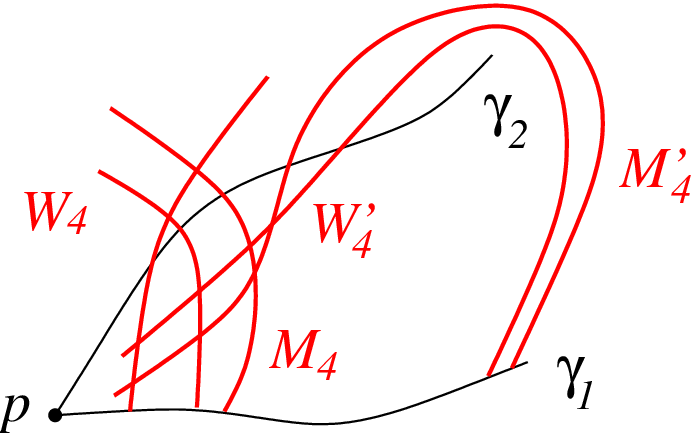}}  \hspace{2cm}
\subfigure[Case 2.]{\label{figca2}\includegraphics[width=0.25\textwidth]{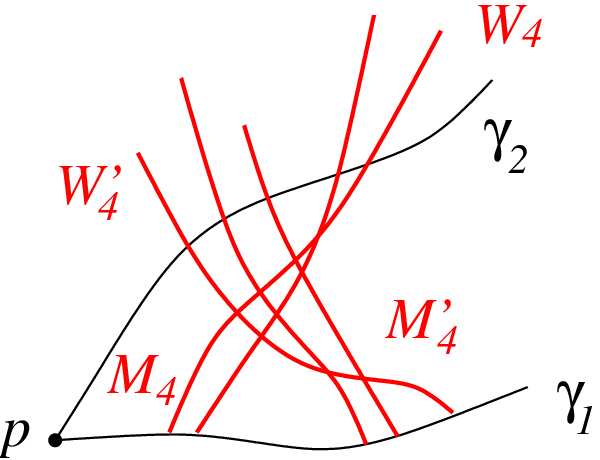}}  \hspace{1cm}
 \caption{In both cases, $\mathcal{W}_4$ is disjoint to $\mathcal{W}_4'.$}
\end{figure}

\noindent See Figure \ref{figca2}.

Hence, by the pseudo-segment condition, every curve in $\mathcal{W}_4$ is disjoint from every curve in $\mathcal{W}'_4$.  By taking arbitrary subsets of $\mathcal{M}_4$ and $\mathcal{B}_2$ of size $\min(|\mathcal{M}_4|,|\mathcal{B}_2|)$, we can apply Theorem \ref{thmdg} to these subsets and obtain subsets $\mathcal{M}_5 \subset \mathcal{M}_4$ and $\mathcal{B}_3\subset \mathcal{B}_2$, each of size $\Omega_{\delta}(n)$, such that either every red curve in $\mathcal{M}_5$ crosses every blue curve in $\mathcal{B}_3$, or every red curve in $\mathcal{M}_5$ is disjoint from every blue curve in $\mathcal{B}_3$.  In the former case, we are done.  Hence, we can assume that we are in the latter case.  

Again, we take an arbitrary subset of $\mathcal{M}'_4$ and $\mathcal{B}_3$ of size $\min(|\mathcal{M}'_4|,|\mathcal{B}_3|)$ and apply Theorem~\ref{thmdg} to $\mathcal{M}'_4$ and $\mathcal{B}_3$, to obtain subsets $\mathcal{M}'_5 \subset \mathcal{M}'_4$ and $\mathcal{B}_4\subset \mathcal{B}_3$, each of size $\Omega_{\delta}(n)$, such that either every red curve in $\mathcal{M}'_5$ crosses every blue curve in $\mathcal{B}_4$, or every red curve in $\mathcal{M}'_5$ is disjoint from every blue curve in $\mathcal{B}_4$.  Again, if we are in the former case, we are done.  Hence, we can assume that we are in the latter case.  Set $\mathcal{R}_5$ be the red curves in $\mathcal{R}$ that corresponds to $\mathcal{M}_5$, and let $\mathcal{R}_5'$ be the red curves in $\mathcal{R}$ that corresponds to $\mathcal{M}'_5$.

 We have $|\mathcal{R}_5|, |\mathcal{R}'_5| = \Omega_{\delta}(n)$, and moreover, we can assume that $|\mathcal{R}_5|= |\mathcal{R}'_5|$.  For each curve $\alpha \in \mathcal{R}_5\cup \mathcal{R}'_5$, and its original partition $\alpha = \alpha_u\cup \alpha_{\ell}$ defined by $\gamma_1$, we have a new partition $\alpha = \alpha_u' \cup \alpha_{\ell}'$ defined by $\gamma_2$, where $\alpha_u' = \alpha_w$ and $\alpha'_{\ell} = \alpha_m \cup \alpha_{\ell}$.  By setting $\mathcal{R}^{(1)} = \mathcal{R}_5\cup \mathcal{R}'_5$, and $\mathcal{B}^{(1)}  = \mathcal{B}_4$, where each curve $\alpha \in \mathcal{R}^{(1)}$ is equipped with the partition  $\alpha = \alpha_u' \cup \alpha_{\ell}'$, we satsify the base case of the statement.

\medskip

 For the inductive step, suppose we have obtained constants $\varepsilon'_{t -1} < \cdots < \varepsilon'_1$ such that the statement follows.  Let $\varepsilon'_t$ be a small constant that will be determined later such that $\varepsilon'_t < \varepsilon'_{t-1}$.  Let $\mathcal{R}$ be a set of $n$ red curves in the plane, all crossed by a curve $\gamma_1$ exactly once, and $\mathcal{B}$ be a set of $n$ blue curves in the plane such that $\mathcal{R}\cup \mathcal{B}\cup \{\gamma_1\}$ is a collection of pseudo-segments. Moreover, $G(\mathcal{R})$ has edge density at least $1 - \varepsilon'_t$ and $G(\mathcal{B})$ has edge density less than $\varepsilon'_t$.  We set  $\delta' < 0$ to be a small constant such that $\varepsilon'_t < \delta' < \varepsilon_{t-1}$.  We repeat the entire argument above, replacing  $\varepsilon'_1$ with $\varepsilon'_t$ and $\delta$ with $\delta'$, to obtain subsets $\mathcal{R}_5,\mathcal{R}'_5 \subset \mathcal{R}$ and $\mathcal{B}_4\subset \mathcal{B}$, each of size $\Omega_{\delta'}(n)$, such that each $\alpha \in \mathcal{R}_5\cup\mathcal{R}'_5$ is equipped with the partition  $\alpha = \alpha_u' \cup \alpha_{\ell}'$, and $\alpha'_{\ell}$ is disjoint to every blue curve in $\mathcal{B}_4$.  Moreover, for $\alpha \in \mathcal{R}_5$ and $\beta \in \mathcal{R}'_5$, where $\alpha = \alpha_u'\cup \alpha_{\ell}'$ and $\beta = \beta_u' \cup \beta_{\ell}'$, $\alpha'_u$ is disjoint to $\beta'_u$.

Since $|\mathcal{R}_5|, |\mathcal{B}_4| \geq \delta_1n$, where $\delta_1$ depends only on $\delta'$, by Theorem \ref{subhom}, $G(\mathcal{R}_5)$ has edge density at least $1  -2\varepsilon'_t/\delta_1^2$ and $G(\mathcal{B}_4)$ has edge density less than $2\varepsilon'_t/\delta_1^2$.    By setting $\varepsilon'_t$ sufficiently small, $G(\mathcal{R}_5)$ has edge density at least $1  -\varepsilon'_{t-1}$, and $G(\mathcal{B}_4)$ has edge density less than $\varepsilon_{t-1}'$.  By averaging, we can find subsets of $\mathcal{R}_5$ and $\mathcal{B}_4$, each of size $\min(|\mathcal{R}_5|,|\mathcal{B}_4|)$ and with densities at least $1-\varepsilon_{t-1}'$ and less than $\varepsilon'_{t-1}$ respectively, and apply induction to these subsets parameter $t' = t-1$, and obtain subsets $\mathcal{R}^{(t-1)} \subset \mathcal{R}_5$, $\mathcal{B}^{(t-1)} \subset \mathcal{B}_4$, each of size $\Omega_{\varepsilon'_{t-1}}(n)$, with the desired properties.  If every red curve in $\mathcal{R}^{(t-1)}$ is disjoint from every blue curve in $\mathcal{B}^{(t-1)}$, or if every red curve in $\mathcal{R}^{(t-1)}$ crosses every blue curve in $\mathcal{B}^{(t-1)}$, then we are done.  Hence, we can assume that each curve $\alpha \in \mathcal{R}^{(t-1)}$ has a partition $\alpha = \alpha_u'' \cup \alpha_{\ell}''$ such that $\alpha_u''$ is a subcurve of $\alpha_u'$, $\alpha''_{\ell}$ is disjoint from every blue curve in $\mathcal{B}^{(t-1)}$, and there is an equipartition

$$\mathcal{R}^{(t-1)} = \mathcal{R}^{(t-1)}_1\cup \cdots \cup \mathcal{R}^{(t-1)}_{2^{t-1}},$$

\noindent such that for $1 \leq i < j \leq 2^{t-1}$, the upper part $\alpha''_u$ of each curve $\alpha \in \mathcal{R}^{(t-1)}_{i}$ is disjoint the upper part $\beta''_u$ of each curve $\beta \in \mathcal{R}^{(t-1)}_{j}$.

Finally, since $|\mathcal{R}'_5|, |\mathcal{B}^{(t-1)}| \geq \delta_2n$, where $\delta_2$ depends only on $\delta'$, by Theorem \ref{subhom}, $G(\mathcal{R}'_5)$ has edge density at least $1  -2\varepsilon'_t/\delta_2^2$ and $G(\mathcal{B}^{(t-1)})$ has edge density less than $2\varepsilon'_t/\delta_2^2$.    By setting $\varepsilon'_t$ sufficiently small, $G(\mathcal{R}'_5)$ has edge density at least $1  -\varepsilon'_{t-1}$, and $G(\mathcal{B}^{(t-1)})$ has edge density less than $\varepsilon_{t-1}'$.   By averaging, we can find subsets of $\mathcal{R}'_5$ and $\mathcal{B}^{(t-1)}$, each of size $\min(|\mathcal{R}'_5|,|\mathcal{B}^{(t-1)}|)$ and with densities at least $1-\varepsilon_{t-1}'$ and less than $\varepsilon'_{t-1}$ respectively, and apply induction to these subsets parameter $t' = t-1$, and obtain subsets $\mathcal{S}^{(t-1)} \subset \mathcal{R}'_5$, $\mathcal{B}^{(t)} \subset \mathcal{B}^{(t-1)}$, each of size $\Omega_{\varepsilon'_{t-1}}(n)$, with the desired properties.  If every red curve in $\mathcal{S}^{(t-1)}$ is disjoint from every blue curve in $\mathcal{B}^{(t)}$, or if every red curve in $\mathcal{S}^{(t-1)}$ crosses every blue curve in $\mathcal{B}^{(t)}$, then we are done.  Hence, we can assume that each curve $\alpha \in \mathcal{S}^{(t-1)}$ has a partition $\alpha = \alpha_u'' \cup \alpha_{\ell}''$ such that $\alpha_u''$ is a subcurve of $\alpha_u'$, $\alpha''_{\ell}$ is disjoint from every blue curve in $\mathcal{B}^{(t-1)}$, and there is an equipartition

$$\mathcal{S}^{(t-1)} = \mathcal{S}^{(t-1)}_1\cup \cdots \cup \mathcal{S}^{(t-1)}_{2^{t-1}},$$

\noindent such that for $1 \leq i < j \leq 2^{t-1}$, the upper part $\alpha''_u$ of each curve $\alpha \in \mathcal{S}^{(t-1)}_{i}$ is disjoint the upper part $\beta''_u$ of each curve $\beta \in \mathcal{S}^{(t-1)}_{j}$.  We then (arbitrarily) remove curves from each part in $\mathcal{R}_i^{(t-1)}$ and $\mathcal{S}_j^{(t-1)}$ such that the resulting parts all have the same size and for

$$\mathcal{R}^{(t)} = \mathcal{R}^{(t-1)}_1\cup \cdots \cup \mathcal{R}^{(t-1)}_{2^{t-1}}\cup \mathcal{S}^{(t-1)}_1\cup \cdots \cup \mathcal{S}^{(t-1)}_{2^{t-1}},$$

\noindent we have $|\mathcal{R}^{(t)}|  = \Omega_{\varepsilon_{t-1}'}(n)$. 
 Then $\mathcal{R}^{(t)}$ and $\mathcal{B}^{(t)}$ has the desired properties.\end{proof}

We now prove the following.

\begin{theorem}\label{thmhvl}
There is an absolute constant $\varepsilon_3 >0$ such that the following holds. Let $\mathcal{R}$ be a set of $n$ red curves in the plane and $\mathcal{B}$ be a set of $n$ blue curves in the plane such that $\mathcal{R}\cup \mathcal{B}$ is a collection of pseudo-segments, and the intersection graph $G(\mathcal{B})$ has edge density less than $\varepsilon_3$, and $G(\mathcal{R})$ has edge density at least $1-\varepsilon_3$.  

Then there are subsets $\mathcal{R}'\subset\mathcal{R}$, $\mathcal{B}'\subset \mathcal{B}$, each of size $\Omega(n)$, such that either every red curve in $\mathcal{R}$ crosses every blue curve in $\mathcal{B}$, or every red curve in $\mathcal{R}$ is disjoint from every blue curve in $\mathcal{B}$.
    
\end{theorem}

\begin{proof}
    
Let $t$ be a fixed large integer such that $2^{-t} < \varepsilon_1$, where $\varepsilon_1$ is defined in Theorem \ref{0a0}. Let $\varepsilon_3$ be a small constant determined later such that $\varepsilon_3 < \varepsilon'_t$, where $\varepsilon'_t$ is defined in Lemma \ref{keylem}.  Recall that $\varepsilon'_t < \varepsilon_1$. Since $G(\mathcal{R})$ has edge density at least $1 - \varepsilon_3$, there is a curve $\gamma_1 \in \mathcal{R}$ such that $\gamma_1$ crosses at least $n/2$ red curves in $\mathcal{R}$.  Let $\mathcal{R}_0\subset \mathcal{R}$ be the red curves that crosses $\gamma_1$.  By Lemma~\ref{subhom}, $G(\mathcal{R}_0)$ has edge density at least $1-8\varepsilon_3$.  By averaging, we can find a subset $\mathcal{B}'\subset \mathcal{B}$ of size $|\mathcal{R}_0|$ whose edge denesity less than $\varepsilon_3$. By setting $\varepsilon_3$ sufficiently small so that $8\varepsilon_3 < \varepsilon_t'$, we can apply Lemma \ref{keylem} to $\mathcal{R}_0$ and $\mathcal{B}'$ with parameter $t$, and obtain subsets $\hat{\mathcal{R}}\subset\mathcal{R}_0$, $\hat{\mathcal{B}}\subset \mathcal{B}$, each of size $\Omega_{\varepsilon_t'}(n)$,  with the desired properties.  If  every red curve in $\hat{\mathcal{R}}$ crosses every blue curve in $\hat{\mathcal{B}}$, or every red curve in $\hat{\mathcal{R}}$ is disjoint from every blue curve in $\hat{\mathcal{B}}$, then we are done.  Therefore, we can assume that each curve $\alpha \in \hat{\mathcal{R}}$ has a partition into two parts $\alpha = \alpha'_{u}\cup \alpha'_{\ell}$ with the properties described in Lemma \ref{keylem}.  Set

$$\mathcal{U} = \{\alpha'_u: \alpha \in \hat{\mathcal{R}}, \alpha = \alpha'_u\cup \alpha'_{\ell}\}\hspace{.5cm}\textnormal{and}\hspace{.5cm}\mathcal{L} = \{\alpha'_{\ell}: \alpha \in \hat{\mathcal{R}}, \alpha = \alpha'_u\cup \alpha'_{\ell}\}.$$

Hence, every curve in $\mathcal{L}$ is disjoint from every curve in $\hat{\mathcal{B}}$, and $G(\mathcal{U})$ has edge density at most $2^{-t} < \varepsilon_1$.    Since $|\hat{\mathcal{B}}| \geq \delta n$, where $\delta$ depends only on $\varepsilon'_t$, by Lemma \ref{subhom}, $G(\hat{\mathcal{B}})$ has edge density at most $2\varepsilon_3/\delta^2$.  By setting $\varepsilon_3$ sufficiently small so that $2\varepsilon_3/\delta^2_0 < \varepsilon_1$, $G(\hat{\mathcal{B}})$ has edge density at most $\varepsilon_1$.  By averaging, we can find subsets of $\mathcal{U}$ and $\hat{\mathcal{B}}$, each of size $\min(|\mathcal{U}|,|\hat{\mathcal{B}}|)$ and with densities at most $\varepsilon_1$, and apply Theorem~\ref{0a0} to these subsets to obtain subsets $\mathcal{U}'\subset \mathcal{U}$ and $\mathcal{B}'\subset \hat{\mathcal{B}}$, each of size $\Omega_{\varepsilon_3}(n)$, such that every curve in $\mathcal{U}'$ is disjoint from every curve in $\mathcal{B}'$, or  every curve in $\mathcal{U}'$ crosses every curve in $\mathcal{B}'$.  By setting $\mathcal{R}'$ to be the red curves in $\mathcal{R}$ corresponding to $\mathcal{U}'$, every red curve in $\mathcal{R}'$ is disjoint from every blue curve in $\mathcal{B}'$, or every red curve in $\mathcal{R}'$ crosses every blue curve in $\mathcal{B}'$, and each subset has size $\Omega_{\varepsilon_3}(n)$. \end{proof}

\subsection{High versus high edge density}

Finally, we consider the case when the intersection graphs $G(\mathcal{R})$ and $G(\mathcal{B})$ both have edge density at least $1-\varepsilon$.  By copying the proof of Theorem \ref{thmhvl}, except using Theorem \ref{thmhvl} (high versus low density) instead of Theorem \ref{0a0} (low versus low density) in the argument, we obtain the following.

\begin{theorem}\label{thmhvh}
There is an absolute constant $\varepsilon_4>0$ such that the following holds. Let $\mathcal{R}$ be a set of $n$ red curves in the plane and $\mathcal{B}$ be a set of $n$ blue curves in the plane such that $\mathcal{R}\cup \mathcal{B}$ is a collection of pseudo-segments, and the intersection graphs $G(\mathcal{B})$ and $G(\mathcal{R})$ both have edge density at least $1-\varepsilon_4$.  

Then there are subsets $\mathcal{R}'\subset\mathcal{R}$, $\mathcal{B}'\subset \mathcal{B}$, each of size $\Omega(n)$, such that either every red curve in $\mathcal{R}$ crosses every blue curve in $\mathcal{B}$, or every red curve in $\mathcal{R}$ is disjoint from every blue curve in $\mathcal{B}$.
    
\end{theorem}

\section{Proof of Theorem \ref{main}}\label{pfmain}

Let $\mathcal{R}$ be a set of $n$ red curves in the plane, and $\mathcal{B}$ be a set of $n$ blue curves in the plane such that $\mathcal{R}\cup\mathcal{B}$ is a collection of pseudo-segments.  Let $\varepsilon$ be a sufficiently small constant such that $\varepsilon < \varepsilon_4 < \varepsilon_3 < \varepsilon_1$, where $\varepsilon_1$ is from Theorem \ref{0a0}, $\varepsilon_3$ is from Theorem \ref{thmhvl}, and $\varepsilon_4$ is from Theorem~\ref{thmhvh}.  We apply Corollary~\ref{lemhom} to both $\mathcal{R}$ and $\mathcal{B}$ and obtain subsets $\mathcal{R}_1\subset \mathcal{R}$ and $\mathcal{B}_1\subset \mathcal{B}$ such that both $G(\mathcal{R}_1)$ and $G(\mathcal{B}_1)$ are $\varepsilon$-homogeneous.  Moreover, we can assume that $|\mathcal{R}_1| = |\mathcal{B}_1|$.

If both $G(\mathcal{R}_1)$ and $G(\mathcal{B}_1)$ have edge densities less than $\varepsilon$, then, since $\varepsilon$ is sufficiently small, we can apply Theorem \ref{0a0} to obtain subsets $\mathcal{R}_2\subset \mathcal{R}_1$ and $\mathcal{B}_2\subset \mathcal{B}_1$, each of size $\Omega_{\varepsilon}(n)$, such that either every red curve in $\mathcal{R}_2$ is disjoint from every blue curve in $\mathcal{B}_2$, or every red curve in $\mathcal{R}_2$ crosses every blue curve in $\mathcal{B}_2$. 

If one of the graphs of $G(\mathcal{R}_1)$ and $G(\mathcal{B}_1)$ has edge density less than $\varepsilon$, and the other has edge density greater than $1-\varepsilon$, then we apply Theorem \ref{thmhvl} to $\mathcal{R}_1$ and $\mathcal{B}_1$ to  obtain subsets $\mathcal{R}_2\subset \mathcal{R}_1$ and $\mathcal{B}_2\subset \mathcal{B}_1$, each of size $\Omega_{\varepsilon}(n)$, such that either every red curve in $\mathcal{R}_2$ is disjoint from every blue curve in $\mathcal{B}_2$, or every red curve in $\mathcal{R}_2$ crosses every blue curve in $\mathcal{B}_2$. 

Finally, if both $G(\mathcal{R}_1)$ and $G(\mathcal{B}_1)$ have edge densities at least $1 - \varepsilon$, then, since $\varepsilon$ is sufficiently small, we can apply Theorem \ref{thmhvh} to obtain subsets $\mathcal{R}_2\subset \mathcal{R}_1$ and $\mathcal{B}_2\subset \mathcal{B}_1$, each of size $\Omega_{\varepsilon}(n)$, such that either every red curve in $\mathcal{R}_2$ is disjoint from every blue curve in $\mathcal{B}_2$, or every red curve in $\mathcal{R}_2$ crosses every blue curve in $\mathcal{B}_2$. $\hfill\square$

\section{Topological graphs with no $k$ pairwise disjoint edges}\label{drawingapp}

In this section, we prove Theorem \ref{disjoint}. 
 Recall that the {\it odd-crossing number} of $G$, denoted by $\ocn(G)$, is the minimum possible number of pairs of edges that cross an odd number of times, over all drawings of $G$. The {\it bisection width} of a graph $G$ is defined as

$$b(G) = \min\limits_{|V_1|,|V_2|\leq 2n/3}|E(V_1,V_2)|,$$

\noindent where the minimum is take over all partitions $V(G) = V_1\cup V_2$, such that $|V_1|,|V_2|\leq 2n/3$.  The following lemma, due to Pach and T\'oth \cite{crossing}, relates the odd-crossing number of a graph to its bisection width.

\begin{lemma}[\cite{pach2}]\label{bisect}
If $G$ is a graph with
 $n$ vertices of degrees $d_1,\ldots,d_n$, then
 $$b(G)\leq O\left(\log n \sqrt{\ocn(G)+\sum_{i=1}^n d_i^2}\right) .$$
 \end{lemma}

\medskip

 \noindent Since all graphs contain a bipartite subgraph with at least half of its edges, Theorem~\ref{disjoint} immediately follows from the following theorem.

 \begin{theorem}
 \label{proof}

If $G = (V,E)$ is an $n$-vertex simple topological bipartite graph with no $k$ pairwise disjoint edges, then $|E(G)| \leq n(\log n)^{O(\log k)}$.

 \end{theorem}

\begin{proof}  Let $c > 0$ and $c_1> 0$ be the absolute constants from Theorem \ref{appdis} and Lemma \ref{bisect} respectively, and let $c_2 > 0$ be a sufficiently large constant that will be determined later.  We define $f(n,k)$ to be the maximum number of edges in an $n$-vertex simple topological bipartite graph with no $k$ pairwise disjoint edges.  We will prove by induction on $n$ and $k$ that

 $$f(n,k) \leq  n(\log n)^{c_2\log k}.$$

 \noindent Clearly, we have $f(2,k) \leq 1$, and by the results on thrackles imply that $f(n,2) \leq 1.4n$.  Assume the statement holds for $n' < n$ and $k' < k$, and let $G$ be an $n$-vertex simple topological bipartite graph with no $k$ pairwise disjoint edges. The proof falls into two cases.

 \medskip

 \noindent \emph{Case 1.}  Suppose there are at least $|E(G)|^2/((2c_1)^2\log^6 n)$ disjoint pairs of edges in $G$.  Let us partition $E(G) = E_1\cup E_2$ into two parts, such that there are at least  $|E(G)|^2/(2(2c_1)^2\log^6 n)$ disjoint pairs in $E_1\times E_2$.  Since $E(G)$ is a collection of pseudo-segments, by Theorem~\ref{appdis}, there exist subsets $E'_1\subset E_1, E'_2\subset E_2$, each of size $|E(G)|/(2c_1\log n)^{6c},$ such that every edge in $E'_1$ is disjoint from every edge in $E'_2$.  Since $G$ does not contain $k$ pairwise disjoint edges, this implies that either $E'_1$ or $E'_2$ does not contain $k/2$ pairwise disjoint edges.  Without loss of generality, suppose $E'_1$ does not contain $k/2$ pairwise disjoint edge. Hence

 $$\frac{|E(G)|}{(2c_1\log n)^{6c}} \leq |E'_1| \leq f(n,k/2).$$

 \noindent By the induction hypothesis, we have

 $$f(n,k/2) \leq n(\log n)^{c_2\log(k/2)}  \leq n(\log n)^{c_2\log(k) - c_2}.$$

 \noindent Hence, for $c_2$ sufficiently large, we have $|E(G)|  \leq n(\log n)^{c_2\log(k)}.$

\medskip

\noindent \emph{Case 2.}  Suppose there are at most $|E(G)|^2/((2c_1)^2\log^6n)$ disjoint pairs of edges in $G$.  In what follows, we will apply a redrawing technique due to Pach and T\'oth in \cite{pach2}.  Since $G$ is bipartite, let $V(G) = V_1\cup V_2$.  We can redraw $G$ such that the vertices in $V_1$ lie above the line $y = 1$, the vertices in $V_2$ lie below the line $y = 0$, the edges in the strip $0 \leq y \leq 1$ are vertical segments, and we have neither created nor removed any crossings.  We then “flip” the half-plane bounded by the $y = 1$ line from left to right about the $y$-axis, and replace the edges in the strip $0\leq y \leq 1$ with straight line segments that
reconnect the corresponding pairs on the line $y = 0$ and $y = 1$.  See Figure \ref{redraw}.

\begin{figure}
\centering
\includegraphics[width=.7\textwidth]{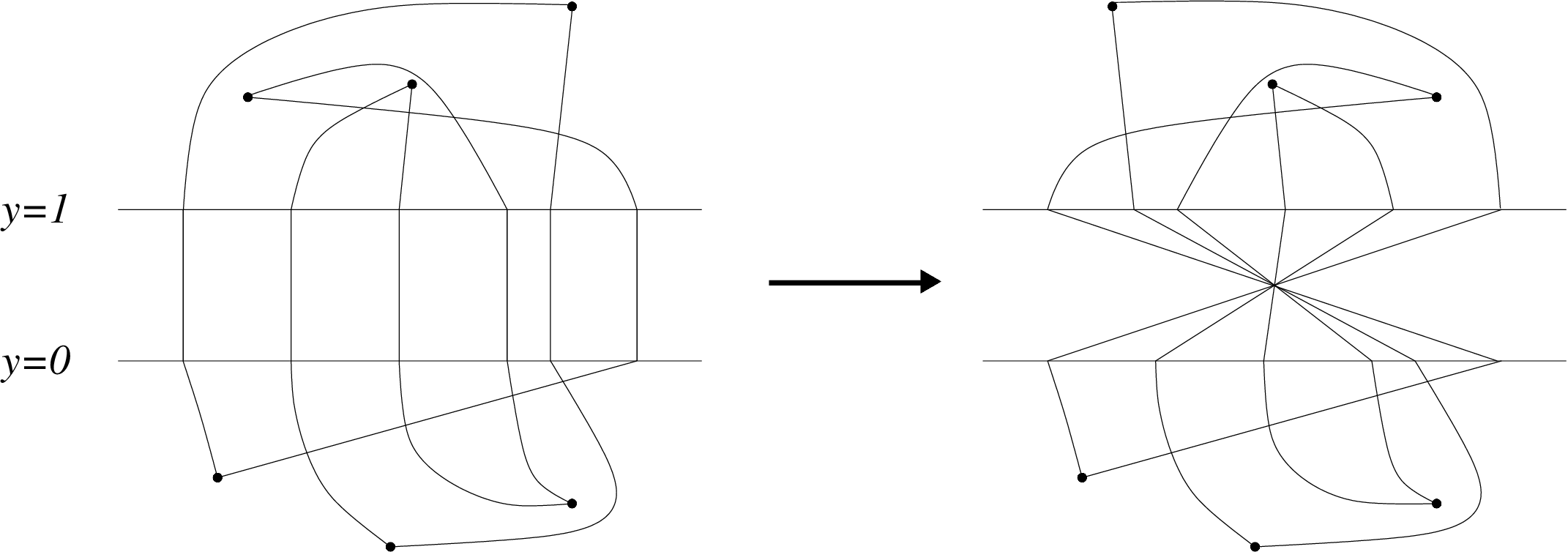}
\caption{Redrawing procedure}
\label{redraw}
\end{figure}

 Notice that if any two edges crossed in the original drawing, then they must cross an even number of times in the new drawing.  Indeed, suppose the edges $e_1$ and $e_2$ crossed in the original drawing.  By the simple condition, they share exactly 1 point in common.  Let $k_i$ denote the number of times edge $e_i$ crosses the strip for $i \in \{1,2\}$, and note that $k_i$ must be odd since $V_1$ lies above the line $y = 1$ and $v_2$ lies below the line $y = 0$. After we have redrawn our graph, these $k_1+k_2$ segments inside the strip will now pairwise cross, creating $\binom{k_1 + k_2}{2}$ crossing points.  Since edge $e_i$ will now cross itself $\binom{k_i}{2}$ times, this implies that there are now

 \begin{equation}
 \label{parity}
\binom{k_1 + k_2}{2} - \binom{k_1}{2} - \binom{k_2}{2} 
 \end{equation}

 \noindent crossing points between edges $e_1$ and $e_2$ inside the strip $0 \leq y \leq 1$.  One can easily check that (\ref{parity}) is odd when $k_1$ and $k_2$ are odd.  Since $e_1$ and $e_2$ had 1 point in common outside of the strip, this implies that $e_1$ and $e_2$ cross each other an even number of times in the new drawing. Moreover, we can easily get rid of self-intersections by making modifications in a small ball at these crossing points.

 Hence, the odd-crossing number in our new drawing is at most the number of disjoint pair of edges in the original drawing of $G$, plus the number of pair of edges that share a common vertex.  Since there are at most
 $$\sum\limits_{v\in V(G)} d^2(v) \leq 2|E(G)|n$$

 \noindent pairs of edges that share a vertex in $G$, this implies

$$\ocn(G) \leq \frac{|E(G)|^2}{(2c_1)^2\log^6n} + 2|E(G)|n.$$

\noindent   By Lemma~\ref{bisect}, there is a partition of the vertex set $V=V_1 \, \dot{\cup} \, V_2$ with $
 |V|/3\leq |V_i|\leq 2|V|/3$, where $i=1,2$, and

 $$|E(V_1,V_2)| \leq b(G) \leq c_1\log n \sqrt{\frac{|E(G)|^2}{(2c_1)^2\log^6 n} +  4n|E(G)|  }.$$

 \noindent  If $|E(G)|^2/((2c_1)^2\log^6 n) \leq  4n|E(G)|$, then for $c_2$ sufficiently large we have $|E(G)| \leq n(\log n)^{c_2\log k}$ and we are done.  Therefore, we can assume

   $$ b(G) \leq c_1\log n \sqrt{\frac{2|E(G)|^2}{(2c_1)^2\log^6 n} } \leq \frac{ |E(G)|}{\log^2 n}.$$

\noindent Let $|V_1| = n_1$ and $|V_2| = n_2$.  Since $n_1,n_2  < n$, by the induction hypothesis, we have

 $$
 \begin{array}{ccl}
 |E(G)| & \leq & b(G)+  n_1(\log n_1)^{c_2\log k} + n_2(\log n_2)^{c_2\log k}\\\\
     & \leq &  \frac{|E(G)|}{ \log^2  n}  + n(\log(2n/3))^{c_2\log k} \\\\
    & \leq &   \frac{|E(G)|}{ \log^2  n}+  n(\log n  - \log(3/2))^{c_2\log k},  \\\\
 \end{array}$$

\noindent which implies
$$|E(G)| \leq n(\log n)^{c_2\log k} \frac{(1 - \log(3/2)/\log n)^{c_2\log k}  }{1 - 1/\log^2n}  \leq n(\log n)^{c_2\log k}.$$\end{proof}

\section{Concluding remarks}

A graph is said to be a \emph{cograph} if it consists of a single-vertex or if it can be obtained from two smaller cographs by either taking their disjoint union or their join. We next define the depth of a cograph. The graph on one vertex has depth $0$. The depth of a cograph which is the disjoint union or join of two smaller cographs is one more than the maximum depth of the two smaller cographs. For example, the depth of a complete or empty graph on $k$ vertices is $\lceil \log_2 k \rceil$. 

The proof of Theorem \ref{disjoint} can be generalized as follows.

\begin{theorem}\label{cograph}
Let $F$ be a cograph with depth $d$.  If $G = (V,E)$ is an $n$-vertex simple topological graph with no matching $M$ whose intersection graph is isomorphic to $F$, then $|E(G)| \leq n(\log n)^{O( d)}$. 
\end{theorem}

\begin{proof}[Proof sketch]
The proof is very similar to the proof of Theorem \ref{proof}.  We proceed by induction on the depth $d$ and $n$. When $d=0$ or $n =1$, the statement is trivial. For the inductive step, let $F = F_1\cup F_2$, where $F_1$ and $F_2$ are two smaller cographs.  If $F$ is the disjoint union of $F_1$ and $F_2$, then we follow the proof of Theorem \ref{proof}.  

If $F$ is the join of $F_1$ and $F_2$, then the proof is again very similar to the proof of Theorem~\ref{proof}.  We consider the two cases when $G$ has at least $|E(G)|^2/(c\log^6n)$ crossing edges, or less than $|E(G)|^2/(c\log^6n)$ crossing edges, where $c$ is a large constant.  In the former case, we apply Theorem~\ref{old} to the edges of $G$ and obtain two subsets $E_1,E_2\subset E(G)$, each of size $|E(G)|/(2c\log n)^{6c'}$ such that every edge in $E_1$ crosses every edge in $E_1$ and $c'$ is an absolute constant.  Hence we can apply induction to each $E_i$ and we are done.    In the latter case, we apply a bisection width formula due to Pach, Shahrokhi, and Szegedy \cite{PSS} and apply induction on $n$.  \end{proof}

A \emph{natural $k$-grid} in a topological graph is a pair of edge sets $E_1, E_2$, such that $E_i$ consists of $k$ pairwise disjoint edges, and every edge in $E_1$ crosses every edge in $E_2$.  In \cite{ackermangrid}, Ackerman and the authors showed that every $n$-vertex simple topological graph with no natural $k$-grid has at most $O(n\log^{4k-6}n)$ edges.  As a corollary to Theorem \ref{cograph}, noting that the balanced complete bipartite graph with $k$ vertices in each part is a cograph with depth $\lceil \log_2 k \rceil +1$, we have the following.

\begin{theorem}
    If $G = (V,E)$ is an $n$-vertex simple topological graph with no natural $k$-grid, then $|E(G)| \leq 3n(\log n)^{O(\log k)}.$
\end{theorem}

\subsection{Further remarks on properties of hereditary families of graphs}

In Section \ref{her}, we proved several properties of hereditary families of graphs are equivalent. Here we extend on these equivalences with some natural variants of these properties. 

It is natural to drop the assumption that the vertex subsets have the same size in the definition of the mighty Erd\H{o}s-Hajnal property. So we say a family $\mathcal{F}$ of graphs has the unbalanced mighty Erd\H{o}s-Hajnal property if there is a constant $c>0$ such that for every graph $G \in \mathcal{F}$ and for all disjoint vertex subsets $A$ and $B$ of $G$, there are subsets $A' \subset A$ and $B' \subset B$ with $|A'| \geq c|A|$ and $|B'| \geq c|B|$ and the bipartite graph between $A'$ and $B'$ is complete or empty. Trivially, if $\mathcal{F}$ has the unbalanced mighty Erd\H{o}s-Hajnal property, then it also has the mighty Erd\H{o}s-Hajnal property. The following proposition shows that something close to a converse also holds. 

\begin{lemma}\label{clonehelp}
If a hereditary family $\mathcal{F}$ of graphs has the mighty Erd\H{o}s-Hajnal property and is closed under cloning vertices, then $\mathcal{F}$ also has the unbalanced mighty Erd\H{o}s-Hajnal property. 
\end{lemma}
\begin{proof}
Since $\mathcal{F}$ has the mighty Erd\H{o}s-Hajnal property, there is a constant $c>0$ such that if $G \in \mathcal{F}$ and $A,B$ are disjoint vertex subsets of $G$ of equal size, then there are $A' \subset A$ and $B' \subset B$ each of size at least $c|A|$ that are complete or empty to each other. 

Let $G \in \mathcal{F}$ and $A,B$ be disjoint vertex subsets of $G$. Every vertex in $A$ we clone $|B|$ times and every vertex in $B$ we clone $|A|$ times to get another graph $G' \in \mathcal{F}$. Let $A^*$ be the vertex subset of $G'$ of clones of vertices in $A$, and $B^*$ be the vertex subset of $G'$ of clones of vertices in $B$, so $|A^*|=|B^*|=|A||B|$. Applying the mighty Erd\H{o}s-Hajnal property to $G'$ and vertex subsets $A^*$ and $B^*$, we get subsets $A' \subset A^*$ and $B^* \subset B$ that are complete or empty to each other and 
$|A'|,|B'| \geq c|A||B|$. Let $A'' \subset A$ be those vertices that contain which has at least one clone in $A^*$ and $B'' \subset B$ be those vertices that contain which has at least one clone in $B^*$. Then $|A''| \geq c|A|$, $|B''| \geq c|B|$, and $A''$ and $B''$ are complete or empty to each other. So $\mathcal{F}$ also has the unbalanced mighty Erd\H{o}s-Hajnal property. 
\end{proof}

Noting that by drawing another pseudo-segment very close to an original pseudo-segment, we see that the family of intersection graphs of pseudo-segments is closed under cloning. So we have the following corollary of Lemma \ref{clonehelp} and Theorem \ref{main}.  

\begin{corollary}
The family of intersection graphs of pseudo-segments has the unbalanced mighty Erd\H{o}s-Hajnal property. 
\end{corollary}

One can similarly define an unbalanced homogeneous density property. 
A family $\mathcal{F}$ of graphs has the \emph{unbalanced homogeneous density property} if there is $C=C(\varepsilon)$ such that for every $G \in \mathcal{F}$ and every pair $A,B \subset V$ of disjoint vertex subsets of $G$ with at least $\varepsilon |A||B|$ edges, there are $A' \subset A$ and $B' \subset B$ with $|A'| \geq \varepsilon^{C}|A|$ and $|B'| \geq \varepsilon^C |B|$ that are complete between them. Trivially, if a family of graphs has the unbalanced homogeneous density property, then it also has the homogeneous density property. In the other direction, just as in the proof of Lemma \ref{clonehelp}, if a hereditary family of graphs has the homogeneous density property and is closed under cloning vertices, then it also has the unbalanced homogeneous density property. In particular, the family of intersection graphs of pseudo-segments has the unbalanced homogeneous density property. 

In the definition of the homogeneous density property, we require a polynomial dependence on the density $\varepsilon$. It turns out that for a hereditary family $\mathcal{F}$ and the family $\overline{\mathcal{F}}$ of complements of graphs in $\mathcal{F}$, having this property is equivalent to the seemingly weaker property that the there is just some dependence on the density $\varepsilon$. We formalize this now. A family $\mathcal{F}$ of graphs has the \emph{weak homogeneous density property} if 
\begin{itemize} 
\item there is a function $\delta:(0,1] \rightarrow (0,1]$ such that 
for every $G \in \mathcal{F}$ and every pair $A,B \subset V$ of disjoint vertex subsets of $G$ of the same size with at least $\varepsilon |A||B|$ edges between them, there are $A' \subset A$ and $B' \subset B$ with $|A'| \geq \delta(\varepsilon)|A|$ and $|B'| \geq \delta(\varepsilon)|B|$ that are complete between them. 
\end{itemize}

We can add to the list of equivalent properties in Theorem \ref{equivalences} that $\mathcal{F}$ and $\overline{\mathcal{F}}$ have the \emph{weak homogeneous density property}. 

\begin{theorem}\label{weakequivalent}
A hereditary family of $\mathcal{F}$ of graphs satisfies $\mathcal{F}$ and $\overline{\mathcal{F}}$ have the homogeneous density property if and only if $\mathcal{F}$ and $\overline{\mathcal{F}}$ have the weak homogeneous density property. 
\end{theorem}

The ``only if'' direction of Theorem \ref{weakequivalent} is trivial. To see the ``if'' direction, observe that the proof of Lemma \ref{hdpimplymighty} that $\mathcal{F}$ and $\overline{\mathcal{F}}$ have the homogeneous density property implies that $\mathcal{F}$ has the mighty Erd\H{o}s-Hajnal property extends to show that $\mathcal{F}$ and $\overline{\mathcal{F}}$ have the weak homogeneous density property implies that $\mathcal{F}$ has the mighty Erd\H{o}s-Hajnal property. Further, Lemma \ref{mightytodensity} implies hereditary $\mathcal{F}$ has the mighty Erd\H{o}s-Hajnal property implies that $\mathcal{F}$ and $\overline{\mathcal{F}}$ have the homogeneous density property. Thus, the ``if'' direction holds as well.

\newpage

\end{document}